\documentclass{m2an}
\newtheorem{lemma}{Lemma}[section] 
\newtheorem{theorem}{Theorem}[section] 
\newtheorem{example}{Example}[section] 
\numberwithin{equation}{section}
\usepackage{color} 
\usepackage{graphicx} 
\usepackage{algorithm}
\usepackage{algorithmic}
\usepackage{epsfig,mathrsfs} 
\usepackage[bf,SL,BF]{subfigure} 
\usepackage{array}
\newcommand{\PreserveBackslash}[1]{\let\temp=\\#1\let\\=\temp}
\newcolumntype{C}[1]{>{\PreserveBackslash\centering}p{#1}}
\newcolumntype{R}[1]{>{\PreserveBackslash\raggedleft}p{#1}}
\newcolumntype{L}[1]{>{\PreserveBackslash\raggedright}p{#1}}

\newcommand{\normmm}[1]{{\left\vert\kern-0.35ex\left\vert\kern-0.25ex\left\vert #1
    \right\vert\kern-0.25ex\right\vert\kern-0.25ex\right\vert}}

\begin{document}
\title{Multiresolution Galerkin method for solving the functional distribution of anomalous diffusion described by time-space fractional diffusion equation}

\thanks{This work was supported by the National Natural Science Foundation of China under Grant  No. 11271173.}
\author{Zhijiang Zhang}\address{School of Mathematics and Statistics,
Gansu Key Laboratory of Applied Mathematics and Complex Systems, Lanzhou University, Lanzhou 730000, P.R. China. zhangzj2011@lzu.edu.cn; dengwh@lzu.edu.cn}
\author{Weihua Deng$^{1}$}
%
%
\begin{abstract}
The functional distributions of particle trajectories have wide applications, including the occupation time in half-space, the first passage time, and the maximal displacement, etc. The models discussed in this paper are for characterizing the  distribution of the functionals of the paths of anomalous diffusion described by time-space fractional diffusion equation. This paper focuses on providing effective computation methods for the models. Two kinds of time stepping schemes are proposed for the fractional substantial derivative. The multiresolution Galerkin method with wavelet B-spline is used for space approximation. Compared with the finite element or spectral polynomial bases, the wavelet B-spline bases have the advantage of keeping the Toeplitz structure of the stiffness matrix, and being easy to generate the matrix elements and to perform preconditioning. The unconditional stability and convergence of the provided schemes are theoretically proved and numerically verified. Finally, we also discuss the efficient implementations and some extensions of the schemes, such as the wavelet preconditioning and the non-uniform time discretization.


\end{abstract}
%
%
\maketitle
\section{Introduction}
A large number of small events, e.g., the motion of pollen grains in water, induce the Brownian dynamics. However, sometimes the rare but large fluctuations result in non-Brownian motion. The small or large fluctuations are reflected by the thin or heavy  tails of the probability density functions (PDF) of the waiting times and jump lengths in the continuous time random walk (CTRW) model. For the CTRW models, if both the first moment of the waiting times and the second moment of the jump lengths are finite, then they describe the Brownian motion. The dynamics considered in this paper can be characterized by CTRW models with the power law waiting times and jump lengths distributions; and both the first moment of the waiting times and the second moment of the jump lengths are divergent. Generally it is called anomalous dynamics.
And the corresponding Fokker-Planck equation is the time-space fractional diffusion equation \cite{Metzler:00}.

The distributions of the functionals of the particle trajectories have wide applications, e.g., the occupation time in half-space, the first passage time, and the maximal displacement, etc. In 1949, inspired by Feynman¡¯s path integrals, Kac derived a Schr\"odinger-like equation for the distribution of the functionals of Brownian motion \cite{Kac:49}. For deep understanding and digging out the potential applications of anomalous diffusion, Carmi, Turgeman, and Barkai derive the models of characterizing the  distribution of the functionals of the paths of anomalous diffusion described by time-space fractional diffusion equation \cite{Carmi:11,Carmi:07,Turgeman:09}, which involve the fractional substantial derivative \cite{Carmi:11, Friedrich:06}. This paper focuses on providing the effective computation methods for the models.
Denoting $x(t)$ as a trajectory of a particle and $U(x)$ a prescribed function, the functional can be defined as $A=\int_0^tU(x(\tau))d\tau$. There are many useful and interesting functionals $A$. For example, we can take $U(x) = 1$ in a given domain and to be zero at other place, which characterizes the  occupation time of a particle in the domain; this functional can be used in kinetic studies of chemical reactions that take place exclusively in the domain. For studying the ergodicity or ergodicity breaking of the non-Brownian motion in inhomogeneous disorder dispersive systems, the $U(x)$ can also be taken as $x$ or $x^2$.

In the CTRW framework, the motion of the particle discussed in this paper is with the PDF of waiting times $ \psi(t)\simeq t^{-(1+\gamma)} \, (\gamma \in (0,1))$ and the PDF of jump lengths $\eta(x)\simeq x^{-(1+\alpha)} \, (\alpha \in (1,2))$. Letting  $G(x,A,t)$  be the joint PDF of finding the particle on $(x,A)$ at time $t$ and $G(x,p,t)$ be the Laplace transform $(U(x)\ge 0)$ $A\to p$ of $G(x,A,t)$,  it satisfies
\begin{equation} \label{eq:subs-1}
\frac{\partial}{\partial t}G(x,p,t)=K_{\gamma,\alpha}\,\nabla_x^{\alpha}\,{}^SD_t^{1-\gamma}G(x,p,t)-pU(x)G(x,p,t),
\end{equation}
being called the forward fractional Feynman-Kac equation \cite{Carmi:11}, where $K_{\gamma,\alpha}>0$ and
the symbol ${}^SD_t^{1-\gamma}$ is the Friedrich's fractional substantial derivative and is equal in Laplace space $t\to s$ to $\left[s+pU(x)\right]^{1-\gamma}$, and $\nabla_x^{\alpha}$ is the Riesz spatial fractional derivative operator defined in Fourier $x\to k$ space as $\nabla_x^{\alpha} \to -\left|k\right|^{\alpha}$. By solving Eq. (\ref{eq:subs-1}) and then performing the inverse Laplace transform, we can get $G(x,A,t)$. If you are only interested in the distribution of the functional $A$, regardless of the current position of the particle $x$, then the integration of $G(x,A,t)$ w.r.t $x$ from $-\infty$ to $+\infty$ is needed. In fact, more convenient way to obtain the distribution of the functional $A$ is to solve the backward fractional Feynman-Kac equation, from which $G_{x_0}(A,t)$ is got. Here $x_0$ is the starting position of the particle; specifying the value of $x_0$ leads to the distribution of the functional $A$. The $G_{x_0}(p,t)$ satifies
\begin{equation} \label{eq:subs-2}
\frac{\partial}{\partial t}G_{x_0}(p,t)=K_{\gamma,\alpha}\,{}^SD_t^{1-\gamma}\,\nabla_{x_0}^{\alpha}G_{x_0}(p,t)-pU(x_0)G_{x_0}(p,t).
\end{equation}
The main efforts of this paper are put into numerically solving Eq. (\ref{eq:subs-2}) with appropriate initial and boundary conditions. Efficient numerical implementations and numerical results are also presented for Eq. (\ref{eq:subs-1}). Most of the time, the analytical solutions of fractional PDEs are not available; occasionally obtained analytical solutions are usually appeared in the form of transcendental functions or infinite series. Therefore, the numerical techniques become essential.

Denoting ${}_0D_t^{-\gamma, \rho}v(t)$ as the tempered Riemann-Liouville fractional integral \cite{Baeumera:10, Cartea:07, Sabzikar:14,Tatar:04}:
\begin{equation}
{}_0D_t^{-\gamma,\rho}v(t)=\frac{1}{\Gamma(\gamma)}\int_0^t e^{-\rho(t-s)}v(s)\frac{ds}{(t-s)^{1-\gamma}},
\end{equation}
it is easy to check that the fractional substantial derivative can be written as
\begin{equation}\label{subsdef-1}
{}^SD_t^{1-\gamma}G_{x_0}(p,t)=\left(\frac{d}{dt}+pU(x_0)\right)\left[e^{-pU(x_0)t}\,{}_0D_t^{-\gamma,0}\left(e^{p U(x_0)t}G_{x_0}(p,t)\right)\right].
\end{equation}
Supposing that $G_{x_0}(p,t)$  has compact support w.r.t. $x_0$ in $\Omega=(a,b)$,  one can rewrite the Riesz space fractional derivative  as \cite{Ding:12,Ervin:07, Xu:14,Yang:10,Zhang:10}
\begin{equation}
\nabla_{x_0}^{\alpha}G_{x_0}(p,t)=\frac{-1}{2\cos(\alpha\pi/2)\Gamma(2-\alpha)}\frac{d^2}{dx_0^2}\int_a^b\left|x_0-\xi\right|^{1-\alpha}G_{\xi}(p,t)d\xi.
\end{equation}
Using the properties of Laplace transform, Eq. (\ref{eq:subs-2}) is mathematically equivalent  to
\begin{equation}\label{eq:subs-3}
{}^SD_{*,t}^{\gamma}G_{x_0}(p,t)=K_{\gamma,\alpha}\nabla_{x_0}^{\alpha}G_{x_0}(p,t),
\end{equation}
where ${}^SD_{*,t}^{\gamma}G_{x_0}(p,t)$ is given as
\begin{eqnarray}\label{subsdef-2}
{}^SD_{*,t}^{\gamma}G_{x_0}(p,t)=e^{-pU(x_0)t}\,{}_0D_t^{-(1-\gamma),0}\left[e^{pU(x_0)t}\left(\frac{d}{dt}+pU(x_0)\right)G_{x_0}(p,t)\right]. \label{eq:caputo}
\end{eqnarray}
The theory analyses for (\ref{eq:subs-3}) apply directly to the more general non-symmetric space fractional derivatives \cite{Deng:08,Ervin:06,Li:10,Liu:04}. And it also displays distinctive computation features that affect the efficiency of the solving process.
With a few calculations, one can show that
\begin{eqnarray}
 &&{}^SD_t^{1-\gamma}G_{x_0}(p,t)={}_0D_t^{1-\gamma,pU}G_{x_0}(p,t):=\frac{e^{-pU(x_0)t}}{\Gamma(\gamma)}\frac{d}{dt}\int_0^te^{pU(x_0)s }G_{x_0}(p,s)\frac{ds}{(t-s)^{1-\gamma}};\\ [5pt]
  &&{}^SD_{*,t}^{\gamma}G_{x_0}(p,t)={}_0D_{*,t}^{\gamma,pU}G_{x_0}(p,t):=\frac{e^{-pU(x_0) t}}{\Gamma(1-\gamma)}\int_0^t\frac{\partial(e^{pU(x_0) s}G_{x_0}(p,s))}{\partial s}\frac{ds}{(t-s)^{\gamma}}, \label{eq:caputosubstential}
\end{eqnarray}
when $pU$ is a positive constant, they become the tempered Riemann-Liouville and Caputo fractional derivative \cite{Baeumera:10, Cartea:07, Sabzikar:14,Tatar:04}, respectively. The derivatives given in (\ref{subsdef-1}) and (\ref{subsdef-2}) are, respectively, called the Riemann-Liouville and Caputo fractional substantial derivative. One should keep in mind that $p$ denotes the parameter in Laplace  space, which  in practice can be  a positive real number or  a complex number with nonnegative real component. If one relaxes the Laplace transform to Fourier transform (when $A\in R$), at this moment $p$ should take complex value  with  zero real component \cite{Carmi:07}. And in the special case $p=0$, Eq. (\ref{eq:subs-2}) and (\ref{eq:subs-3}) reduce to the usual time-space fractional diffusion equations, which have been well studied by many authors \cite{Deng:08,Li:10,Lin:11, Zhang:11, Zhang:14}.

For the differential equations with time and/or space fractional derivatives, instead of fractional substantial derivatives, their numerical schemes have been well developed. The fractional discrete convolution based on the fractional version of the backward differential formula (FBDF) \cite{ Lubich:88, Lubich:04} and the directly numerical integral based on  product integration rules (PI) \cite{Diethelm:97,Galeone:06} are frequently  used to discretize the time fractional derivative; the former generally uses uniform meshes but the latter can easily adapt to non-uniform grids \cite{Zhang:14} to overcome the initial singularity or to reduce the computation cost.
More recently, E. Cuesta at al \cite{Cuesta:06} also developed convolution quadrature  for the fractional diffusion-wave equations, with the highlight  to render the time integration of second order accuracy under realistic and weak regularity assumptions.  K.  Mustapha et al. \cite{Mustapha:13} proposed non-uniform time partition discontinuous finite element method  for the fractional sub-diffusion and wave equations. Y. Lin et al. \cite{Lin:11} presented
the $L_1$ approximation  and its variable time step version are also checked by \cite{Yuste:12,Zhang:14}. The time discrete schemes via Laplace transforms  are  proposed by  W. Mclean et al. \cite{Mclean:06}. In this paper, we provide the time-stepping schemes to approximate the time-space coupled fractional substantial derivative. Taking into account the regularity of the solution, the balance between accuracy, numerical stability, and computation complexity, the following analysis will be restricted to the schemes with convergence order no more than two.


To deal with the Riemann-Liouville space fractional derivative, the finite difference methods \cite{Meerschaert:04,Sousa:11,Tian:15,Yuste:05,Zhao:14}, the continuous/discontinuous finite element methods \cite{ Deng:08, Ervin:07,Mustapha:13, Xu:14} and the spectral/spectral element methods \cite{Li:09,Li:10,Zayernouri:13,Zayernouri:14} are also well developed nowadays. Recently, Wang et al. \cite{Wang:14, Wang:15} proposed a Petrov-Galerkin finite element method to deal with the fractional boundary value problem  with variable-coefficient.  H. Hejazi et al.\cite{Hejazi:14} studied the finite volume method for solving the fractional advection-dispersion equation. Fast algorithms was also considered \cite{Pang:12,Wang:12} to solve  the corresponding algebraic equations. To the best of our knowledge, until so far lack the references of using wavelet method to solve fractional PDEs, although the wavelet numerical methods for classical PDEs have been well developed; see, e.g., \cite{Cai:96,Cohen:00,Dahmen:07,Dahmen:97,Jia:11,Urban:09,Vasilyev:05,Wang:96}.
Comparing to the traditional finite element or polynomial approximation, the stiffness matrix with wavelet scaling bases is  Toeplitz; the reason is that the bases are obtained by dilating and translating of a single function, which allows one to generate the stiffness matrix with the storage and computation cost only $\mathcal{O}(N)$. Meanwhile, unlike the finite difference method \cite{Zhao:15}, the special boundary correction weights introduced for keeping the high order approximation  lead to difficulty in numerical analysis. Moreover, combining with the FFT and FWT, one can solve the corresponding matrix equation with the computation complexity $\mathcal{O}(N\log N)$;  one can also design effective adaptive algorithms using the multiscale bases without the difficulty to have local a posteriori error estimate; see the details in \cite{Zhang:15}. In this paper, we develop the multiresolution Galerkin method (MGM) for space discretization.

The contents are organized as follows. In Section 2, some preliminary lemmas and related background knowledge are introduced. The  numerical discrerization schemes and the corresponding stability and convergence analyses are provided and detailedly proved in Section 3. In Section 4, we discuss the effective implementations and three numerical examples are computed, which well verify the theoretical results given in Section 3. And we conclude the paper with some remarks in the last section.

\section{Preliminaries}
In this section, we collect/derive some essential properties or schemes for the fractional substantial derivatives
and  wavelets.  In the following, $(\cdot,\cdot)$ and $\|\cdot\|$ denote the $L_2$ inner product and norm, respectively. For $\mu$ being an integer, $H^{\mu}(\Omega)$ denotes the usual Sobolev space, otherwise $H^{\mu}(\Omega)$ denotes the fractional Sobolev spaces with the norms given in \cite{Ervin:06, Wang:14}. $H^{\mu}_0(\Omega)$ is the completion of $C_0^{\infty}(\Omega)$ w.r.t. the norm $\|\cdot\|_{H^{\mu}(\Omega)}$, and $H^{-\mu}(\Omega)$ denotes its dual space.

\begin{lemma}\label{lemma:2.1}
Suppose that $0<\gamma<1$ and $v(t)\in H^1(0,T)$. Then
\begin{equation} \label{eq:2.1}
{}^SD_{*,t}^{\gamma}v(t)={}^SD_t^{\gamma}\Big[v(t)-e^{-pU t}v(0)\Big]={}^SD_t^{\gamma}v(t)-\frac{v(0)e^{-pU t}}{t^{\gamma}\Gamma(1-\gamma)}.
\end{equation}
Moreover, if there exists $m\in \mathcal{N}^+$, such that $v(t)$ is $m+1$-times differentiable,  and denotes $D_t^m=\left(\frac{d}{dt}+pU\right)^m$, then
\begin{equation} \label{eq:2.2}
v(t)=\sum_{k=0}^{m}\frac{e^{-pU (t-a)}(t-a)^{k}}{\Gamma(k+1)}(D_t^kv(t))\Big|_{t=a}+e^{-pU t}{}_aD_t^{-(m+1),\,0}\,\frac{d^{m+1}\left(e^{pU t}v(t)\right)}{d t}.
\end{equation}
\end{lemma}
\begin{proof}
Noting that
\[{}_0D_{*,t}^{\gamma,0}(e^{pU t}v(t))={}_0D_{*,t}^{\gamma,0}(e^{pU t}v(t)-v(0))={}_0D_t^{\gamma,0}(e^{pU t}v(t)-v(0)),\]
then (\ref{eq:2.1}) is obtained.
Using the Taylor expansion of $e^{pU t}v(t)$ at $ t=a$
and applying
\begin{equation}
\frac{d}{d t}\left(e^{pU t}v(t)\right)=e^{pU t}\left(\frac{d}{dt}+pU\right)v(t),
\end{equation}
arrives at (\ref{eq:2.2}).
\end{proof}

From Lemma \ref{lemma:2.1}, it can be noted that it is  $e^{-pU t} v(0)$ rather than $v(0)$ that lies in the kernel space of the Caputo fractional substantial operator. And a more general conclusion can be stated as \cite{Chen:15}: if $r\ge 0$, $ n=\left\lceil\gamma \right\rceil$, and $v$ possesses $(n-1)$-th derivative at $ 0$, then there exists
\begin{equation}
{}^SD_{*,t}^{\gamma}v(t)={}^SD^{\gamma}_t\Big[v(t)-T_{n-1}[v;0]\Big],\quad T_{n-1}[v;0]=\sum\limits_{k=0}^{n-1}\frac{e^{-pU t}t^{k}}{\Gamma(k+1)}(D_t^kv(t))\Big|_{t=0_+}.
\end{equation}

\begin{lemma} [see \cite{Li:10}] \label{lemma:2.2}
If $0<\gamma<1/2, v\in L^2(0,T)$, or if $1/2\le\gamma <1,\, v\in H^s(0,T), \gamma-1/2<s<1/2$,  then it holds that
\begin{equation}
{}_0D_t^{-(1-\gamma),0}v(t)\Big|_{t=0_+}=0.
\end{equation}
\end{lemma}

By Lemma \ref{lemma:2.2}, for $0<\gamma<1$, it is easy to find that
\begin{eqnarray} \label{eq:Laplace}
&&\mathbb{L}\left\{{}^SD_t^{\gamma}v(t);s\right\}=\left(s+pU\right)^{\gamma} \mathbb{L}\left\{v(t);s\right\},\\ [4pt]
&&\mathbb{L}\left\{{}^SD_{*,t}^{\gamma}v(t);s\right\}=\left(s+pU\right)^{\gamma} \mathbb{L}\left\{v(t);s\right\}-\left(s+pU\right)^{\gamma-1}v(0),
\end{eqnarray}
where $\mathbb{L}\{k(t),s\} $ denotes the Laplace transform of $k(t)$, i.e.,
\begin{equation}
\mathbb{L}\left\{k(t),s\right\}:=\hat{k}(s)=\int_0^\infty e^{-st}k(t)\,\mathrm{d}t \quad \mbox{for} \,\,\Re \left(s\right)>0.
\end{equation}

With a first look at the Caputo fractional substantial derivative (\ref{eq:caputosubstential}), it is an integro-differential operator. And an intuitive idea is to use the product integral rule (PI) to get a numerical approximation.  Let $0=t_0<\cdots<t_{n-1}<\cdots<t_N=T$  and rewrite the Caputo fractional substantial derivative as
\begin{equation}
{}^SD_{*,t}^{\gamma}v(t_n)=\frac{e^{-pU t_n}}{\Gamma(1-\gamma)}\sum_{j=0}^{n-1}\int_{t_j}^{t_{j+1}}(t_n-s)^{-\gamma}\frac{d(e^{pU s}v(s))}{ds}\,\mathrm{d}s.
\end{equation}
Replacing $e^{pU s}v(s)$ in each time subinterval $[t_j,t_{j+1}]$ by its first-degree polynomial interpolation
\begin{equation}
T_j(s)=e^{pU t_{j+1}}v(t_{j+1})+\frac{s-t_{j+1}}{\tau_j}\left(e^{pU t_{j+1}}v(t_{j+1})-e^{pU t_j}v(t_j)\right) \quad {\rm with ~} \tau_j=t_{j+1}-t_j,
\end{equation}
to produce the discrete form
\begin{equation}\label{eq:PI}
{}^SD_{*,t}^{\gamma}v(t_n)=\Gamma(1-\gamma)^{-1}\sum_{j=0}^{n}e^{pU(t_j-t_n)}k^{\gamma}_{n-j}v(t_j)+R^n,
\end{equation}
where
 \begin{eqnarray}
&&\quad k^{\gamma}_0=\int_{t_{n-1}}^{t_n}(t_n-s)^{-\gamma}\tau_{n-1}^{-1}ds,\quad k^{\gamma}_n=-\int_{t_0}^{t_{1}}(t_n-s)^{-\gamma}\tau_0^{-1}ds,\\[3pt]\nonumber
&&\quad k^{\gamma}_j= \int_{t_{n-j-1}}^{t_{n-j}}(t_n-s)^{-\gamma}\tau_{n-j-1}^{-1}ds-\int_{t_{n-j}}^{t_{n-j+1}}(t_n-s)^{-\gamma}\tau_{n-j}^{-1}ds,~ {\rm for}~ j=1,\cdots,n-1,\label{eq:xishu}
\end{eqnarray}
and it holds that
\begin{eqnarray} \label{eq:PIError}
\left|R^n\right|&\le&\frac{e^{-pU t_n}}{\Gamma(1-\gamma)}\sum_{j=0}^{n-1}\left|\int_{t_j}^{t_{j+1}}\frac{d(e^{pU s}v(s)-T_j(s))}{ds}\frac{\mathrm{d}s}{(t_n-s)^{\gamma}}\right|\\ \nonumber
   &\le&\frac{C}{\Gamma(1-\gamma)}\left\{\gamma\sum_{j=0}^{n-2}\int_{t_j}^{t_{j+1}}\frac{(s-t_j)(t_{j+1}-s)ds}{{(t_n-s)^{\gamma+1}}}+
\int_{t_{n-1}}^{t_n}\frac{2s-t_{n-1}-t_n}{(t_n-s)^{\gamma}}ds\right\}\\\nonumber
&\le&\frac{C\gamma}{\Gamma(1-\gamma)}\sum_{j=0}^{n-2}\tau_j^2\int_{t_j}^{t_{j+1}}\frac{ds}{{(t_n-s)^{\gamma+1}}}+\frac{2C}{\Gamma(2-\gamma)}\int_{t_{n-1}}^{t_n}\frac{ds}{{(t_n-s)^{\gamma-1}}}\\ \nonumber
&\le&\frac{C\gamma}{\Gamma(1-\gamma)}\sum_{j=0}^{n-2}\tau_j^2\int_{t_j}^{t_{j+1}}\frac{ds}{{(t_n-s)^{\gamma+1}}}+\frac{2C}{\Gamma(3-\gamma)}\tau_{n-1}^{2-\gamma},\\ \nonumber
\end{eqnarray}
 with $C=\max\limits_{0\le t\le t_n}\left|e^{pU (t-t_n)}D_t^2v(t) \right|$. Obviously, $\left|R^n\right|=\mathcal{O}(\tau^{2-\gamma})$ for $\tau_j=\tau$.

With a further look at the Caputo fractional substantial derivative (\ref{eq:caputosubstential}), it can also be seen as a convolution integral given as the following form
\begin{equation}
 e^{-pU t}\left(k*g\right)=e^{-pU t}\int_0^tk(t-z)g(z)dz,
\end{equation}
with $k(t)=\frac{t^{-\gamma}}{\Gamma(1-\gamma)}$ and $g(t)=\frac{d(e^{pU t}v(t))}{dt}$. So the discrete convolution developed in \cite{Lubich:88,Lubich:04}
 may be used to get a type of approximation. Recalling the derivation process of the PI scheme and the structure of $k_j^{\gamma}$, we can just replace the integral weights of the intervals $[t_j,t_{j+1}], \,j=0,1,\cdots,n-1$ by the discrete convolution weights, such as the first order FBDF weights $\pi_j$ determined by the generating function  $\sum_{j=0}^{\infty}\pi^{\gamma-1}_j\xi^j=\hat{k}\left(\frac{1-\xi}{\tau}\right)$ to produce an approximation also with form (\ref{eq:PI}), but now
\begin{eqnarray}
&&\Gamma(1-\gamma)^{-1}k^{\gamma}_0=\tau^{-1}\pi^{\gamma-1}_0,\quad \Gamma(1-\gamma)^{-1}k^{\gamma}_n=-\tau^{-1}\pi^{\gamma-1}_{n-1},\\[4pt]\nonumber
&&\Gamma(1-\gamma)^{-1}k^{\gamma}_j=\tau^{-1}(\pi^{\gamma-1}_{j}-\pi^{\gamma-1}_{j-1}),\quad j=1,2,\cdots,n-1.
\end{eqnarray}
Of course the uniform stepsize $\tau$ must be used.  Now let us determine the local truncation error. In fact, from
\begin{equation}
\left(\frac{1-\xi}{\tau}\right)^{\gamma}=(1-\xi)\tau^{-1}\left(\frac{1-\xi}{\tau}\right)^{\gamma-1},
\end{equation}
one has
\begin{equation}
\sum_{j=0}^{+\infty}w_j^{\gamma}\xi^j=(1-\xi)\tau^{-1}\sum_{j=0}^{+\infty}\pi_j^{\gamma-1}\xi^j=\tau^{-1}\pi_0^{\gamma-1}+\sum_{j=1}^{+\infty}\tau^{-1}(\pi_j^{\gamma-1}-\pi_{j-1}^{\gamma-1})\xi^j,
\end{equation}
which actually means  that $\tau^{-1}\pi_j^{\gamma-1}=\sum_{l=0}^{j}w_l^{\gamma}$ for $j=0,1,\cdots$. Hence
\begin{eqnarray}\label{eq:FLMMs}
{}^SD_{*,t}^{\gamma}v(t_n)&=&\Gamma(1-\gamma)^{-1}\sum_{j=0}^{n}k^{\gamma}_{n-j}v(t_j)+R^n\\\nonumber
&=&\sum_{j=1}^{n}w_{n-j}^{\gamma}e^{-pU(n-j)\tau}v(t_j)-e^{-pU n\tau}\sum_{l=0}^{n-1}w_l^{\gamma}v(0)+R^n\\\nonumber
&=&\sum_{j=0}^n w_{n-j}^{\gamma}e^{-pU(n-j)\tau}\left(v(t_j)-e^{-pU j\tau}v(0)\right)+R^n\\
&=&e^{-pU t_n}\sum_{j=0}^n w_{n-j}^{\gamma}\left(e^{pU t_j}v(t_j)-v(0)\right)+R^n.\nonumber
\end{eqnarray}
Noting that $e^{pU t}v(t)-v(0)\big|_{t=0}=0$, by the well known Gr\"unwald-Letnikov discrete approximation to the  Riemann-Liouville operators \cite{Meerschaert:04} and Lemma  \ref{lemma:2.1}, one has
\begin{equation}
R^n={}^SD_{*,t}^{\gamma}v(t_n)-e^{-pU t_n}\sum_{j=0}^n w_{n-j}^{\gamma}\left(e^{pU t_j}v(t_j)-v(0)\right)=\mathcal{O}(\tau).
\end{equation}

For  space discretization, let $[x_0,\ldots,x_d]f$ denote the $d$-th order {\em divided difference} of $f$ at the points $x_0,\ldots, x_d$, $t_+^l:=(\max\{0,t\})^l$, and choose the Schoenberg sequence of knots $\boldsymbol{t}^j$ as
  \begin{equation}
	{\boldsymbol{t}^j}:=\{\underbrace{0,\ldots,0}_{d}, 2^{-j},\ldots,1-2^{-j},\underbrace{1,\ldots,1}_{d}\},
	\end{equation}
to define the scaling function sets $\Phi_j=\Big\{ \phi_{j,k}, k\in\triangle_j=\left\{1,\ldots,2^j+d-3\right\}\Big\}$, where
	\begin{equation}\label{scale_base}
	\phi_{j,k}(x):=2^{\frac{j}{2}}(t_{k+d+1}^j-t_{k+1}^j)[t_{k+1}^j,\ldots,t_{k+d+1}^j](t-x)_+^{d-1}.
	\end{equation}
Then  $S_j={\rm span}\{\Phi_j\}$ form a multiresolution analysis (MRA) of $L_2(I),\,I=(0,1)$. The system $\Phi_j$ is uniformly local and locally finite, i.e., ${\rm diam} ({\rm supp}\phi_{j,k})\stackrel{<}{\sim}2^{-j}$ and $ \#\{\phi_{j,k}: {\rm supp}\phi_{j,k}\cap {\rm supp}\phi_{j,i}\}\stackrel{<}{\sim}1$; it forms a stable Riesz basis of $S_j$, i.e.,
	\begin{equation}
 c_{\Phi}\|{\boldsymbol{c}_j}\|_{l_2(\triangle_j)}\le\Big\|\sum_{\lambda\in\Delta_j}c_{j,k}\phi_{j,k}\Big\|_{L_2(\Omega)}
    \le C_{\Phi}\|{\boldsymbol{c}_j}\|_{l_2(\triangle_j)};
    \end{equation}
and $S_j$ satisfies the Jackson and Bernstein estimates, i.e.,
	\begin{eqnarray} \label{eq:Jackson}
	  & &\inf_{v_j\in S_j}\left\|v-v_j\right\|_{L_2(\Omega)}\stackrel{<}{\sim}2^{-jd}\left\|v\right\|_{\mathcal{H}^d(I)} \quad \forall v\in \mathcal{H}_0^d(I),\\
	  & &\left\|v_j\right\|_{\mathcal{H}^s(\Omega)}\stackrel{<}{\sim}2^{js}\left\|v_j\right\|_{L_2(I)}, \quad \forall  v_j\in S_j,\ 0\le s\le \gamma,
	\end{eqnarray}
where $\gamma:=\sup\{\nu\in \mathcal{R}:v_j\in \mathcal{H}^\nu(I),~~ \forall v_j\in S_j\}$ and by $A\stackrel{<}{\sim}B$ we mean that $A$ can be bounded by a multiple of $B$, independent of the parameters they may depend on.

By the properties of the Riesz basis, there exists a dual MRA  sequence $\{\tilde{S}_j\} $ with $\tilde{S}_j={\rm span}\{\widetilde{\Phi}_j\}$ and the corresponding biorthogonal wavelets $\Psi_j=\{\psi_{j,k},k\in \nabla_j\}$. The wavelet can well characterize the space (norm equivalence) \cite{Cohen:00, Dahmen:97,Primbs:06}:
there exists $\tilde{\sigma}, \sigma >0$ s.t. for $ \mathcal{H}_0^s(I), s\in(-\tilde{\sigma},\sigma)$:
\begin{equation} \label{normequiv}
	\quad\Big\|\sum_{j\ge J_0-1}\sum_{k\in\nabla_j}d_{j,k}\psi_{j,k}\Big\|_{\mathcal{H}^s(I)}^2\sim
	                                  \sum_{j\ge J_0-1}\sum_{k\in \nabla_{j}}2^{2js}\big|d_{j,k}\big|^2,
\end{equation}
where $\psi_{J_0-1,k}:=\phi_{J_0,k}, \nabla_{J_0-1}:=\Delta_{J_0}, d_{J_0-1,k}:=c_{J_0,k}$, $J_0$ denotes the lowest level.
The range of $s$ is determined by the basic properties (\ref{eq:Jackson}). It  means that $\bigcup_{j=J_0-1}^\infty2^{-js}\Psi_j$ is a Riesz basis of $\mathcal{H}_0^s(I)$. In addition, there exist refinement matrixes $M_{j,0}$ and $M_{j,1}$ satisfying
 \begin{equation}\label{Refine-raltion}
	\Phi_j^T=\Phi_{j+1}^TM_{j,0}, \quad \Psi_j^T=\Phi_{j+1}^TM_{j,1};
	\end{equation}
 and for any $u_J\in S_J$,
	\begin{equation}\label{eq:2.15}
	u_J=\sum_{k\in \triangle_J}c_{J,k}\phi_{J,k}=\sum_{k\in\triangle_{J_0}}c_{J_0,k}\phi_{J_0,k}+\sum_{j=J_0}^{J-1}\sum_{k\in\nabla_j}d_{j,k}\psi_{j,k}.
	\end{equation}
Using the FWT, the overall computation complexity is $\mathcal{O}(2^J)$ for the transition of coefficients, viz,
if denote ${\boldsymbol {\mathrm c}}_j=(c_{j,k})_{k\in \triangle_j}$, $ {\boldsymbol{\mathrm d}_j}=(d_{j,k})_{k\in\nabla_j}$
 and $ {\boldsymbol d}_J=({\boldsymbol{\mathrm c}_{J_0},{\mathrm d}_{J_0},\ldots,{\mathrm d}_{J-1}})$, then
\begin{equation}\label{eq:2.16}
{\boldsymbol c}_J=M{\boldsymbol d}_J,
\end{equation}
where
	\begin{equation}\label{eq:2.21}
	M=\left(\begin{array}{cc}M_{J-1}&0\\ 0&I_{J-1}\end{array}\right)\left(\begin{array}{cc}M_{J-2}&0\\0&I_{J-2}\end{array}\right)\cdots
	                                                                 \left(\begin{array}{cc}M_{J_0}&0\\0&I_{J_0}\end{array}\right),
	\end{equation}
with $I_j, j=J_0,\cdots,J-1$, being the identity matrix, and $M_j=(M_{j,0},M_{j,1})$.
For more details, refer to  \cite{Cohen:00,Primbs:06,Urban:09,Zhang:15}.

Finally, we point out that there exists a biorthogonal projector:
 \begin{equation}
   \Pi_j: L_2(\Omega)\to S_j, \qquad \Pi_ju:=\sum_{k\in\triangle_{j}}\left(u,\tilde{\phi}_{j,k}\right)\phi_{j,k},
   \end{equation}
which satisfies $\Pi_{j+1}\Pi_j=\Pi_j\Pi_{j+1}=\Pi_j$.  For every $u\in H_0^{s}(I)\cap H^{\gamma}(I)$, it holds that
\begin{equation}\label{project}
  \left\|u-\Pi_ju\right\|_{\mathcal{H}^s(I)}\stackrel{<}{\sim}2^{j(s-r)}\left\|u\right\|_{\mathcal{H}^\gamma(I)},\quad 0\le s<r\leq d.
\end{equation}
The scaling and wavelet functions in $\Omega$ can be obtained through an affine map, and they can also be extended  to high dimensional cases by tensor product, keeping the same properties as the one dimensional version.

Define the bilinear form $B(\cdot,\cdot): H^{\frac{\alpha}{2}}_0(\Omega)\times H^{\frac{\alpha}{2}}_0(\Omega)\to R$ as
\begin{equation}
B(u,v):=\frac{\left({}_aD_x^{\frac{\alpha}{2},0}u,{}_xD_b^{\frac{\alpha}{2},0}v\right)+\left({}_xD_b^{\frac{\alpha}{2},0}u,{}_aD_x^{\frac{\alpha}{2},0}v\right)}{2\cos(\alpha\pi/2)},
\end{equation}
where ${}_xD_1^{\frac{\alpha}{2},\rho}v$ is the right tempered Riemann-Liouville fractional derivative given by \begin{eqnarray}
{}_xD_b^{\frac{\alpha}{2},\rho}v(x)=\frac{-e^{\rho x}}{\Gamma(1-\alpha/2)}\frac{d}{dx}\int_x^b  e^{\rho \xi}v(\xi)\frac{d\xi}{(\xi-x)^{\alpha/2}}.
\end{eqnarray}
And $B(\cdot,\cdot)$ is continuous and coercive \cite{Ervin:07}, i.e., there exist constants $ C_1$ and $C_2$ such that
\begin{equation}\label{eq:coercive}
	   |B(u,v)|\le C_1\|u\|_{H^\frac{\alpha}{2}(\Omega)}\|v\|_{H^\frac{\alpha}{2}(\Omega)}, \quad C_2\|u\|_{H^\frac{\alpha}{2}(\Omega)}^2\le  B(u,u).
	\end{equation}
If for every $v\in S_J\subset H^{\frac{\alpha}{2}}_0(\Omega)$, $ \Lambda u\in S_J$ satifies $B(u-\Lambda u,v)=0$, then it follows that \cite {Ervin:06,Zhang:15}
\begin{equation}
\left\|u-\Lambda u\right\|_{H^\frac{\alpha}{2}(\Omega)}\stackrel{<}{\sim }\left\|u-\Pi_ju\right\|_{H^{\frac{\alpha}{2}}(\Omega)}\stackrel{<}{\sim } 2^{J(\alpha/2-r)}\left\|u\right\|_{H^{r}(\Omega)}.
\end{equation}

\section{Numerical schemes and related analyses}
Based on the model (\ref{eq:subs-2}) itself and its equivalent form, we derive two numerical schemes and perform the theoretical analysis for their semi and full discrete forms. We rewrite (\ref{eq:subs-2}) as
\begin{equation}\label{eq:model1}
\left\{
\begin{array}{ll}
\frac{\partial}{\partial t}G_{x}(p,t)=K_{\gamma,\alpha}\,{}^SD_t^{1-\gamma}\,\nabla_{x}^{\alpha}G_{x}(p,t)-pU(x)G_{x}(p,t),&x\in \Omega,\,0<t\le T,\\[5pt]
~ G_x(p,0)=g(x,p), & x\in \Omega,
\end{array}
\right.
\end{equation}
subjecting to homogeneous boundary conditions.

\subsection{Numerical schemes $\mathrm{I}$} 
For making a more general discussion, we add a force term $f(x,p,t)$ at the right hand side of (\ref{eq:model1}); of course, $f$ can be zero.
The weak form of the model is
\begin{equation}\label{eq:3.1}
\left(D^1_t(G_{x}(p,t)), v\right)=\left\langle K_{\gamma,\alpha}\, {}^SD_t^{1-\gamma}\,\nabla_{x}^{\alpha} G_x(p,t),v\right\rangle+\left(f,v\right),
\end{equation}
where $v$ belongs to appropriate space. Denote $G_J$ and $G^n_J$ as the approximation of $G_{x}(p,t)$ and $G_{x}(p,t_n)$, respectively, in the space $S_J$. Let $\mathcal{L}_1G_J^n$ and $\mathcal{L}_2G_J^n$ be the discretization of $
D_t^1 (G_J) \big|_{t=t_n}$ and ${}^SD_t^{1-\gamma}\,G_J$, respectively, viz, for the FBDF time discretization,
\begin{equation}\label{eq:L11}
\mathcal{L}_1G_J^n=(G_J^n-e^{-pU \tau}G_J^{n-1})/\tau, \quad \mathcal{L}_2G_J^n= e^{-pU n\tau}\sum_{j=0}^n w_{n-j}^{1-\gamma}\left(e^{pU j\tau}G_J^j-G_J^0\right)+\frac{e^{-pUn\tau}G_J^0}{t_n^{1-\gamma}\Gamma(\gamma)};
\end{equation}
and for  the PI time discretization,
\begin{eqnarray}\label{eq:L1-1}
&&\mathcal{L}_1G_J^n=\left\{\begin{array}{ll}(G_J^n-e^{-pU \tau}G_J^{n-1})/\tau, &n=1,\\[5pt]
                                 (\frac{3}{2}G_J^n-2e^{-pU \tau}G_J^{n-1}+\frac{1}{2}e^{-2pU \tau}G_J^{n-2})/\tau,& n\ge 2,\end{array} \right.\,\,
                                 \\
&&\mathcal{L}_2G_J^n=\frac{1}{\Gamma(\gamma)}\sum_{j=0}^{n}e^{pU(j-n)\tau}k^{1-\gamma}_{n-j}G_J^j+\frac{e^{-pUn\tau}G_J^0}{t_n^{1-\gamma}\Gamma(\gamma)}, \label{eq:L1-21} \end{eqnarray}
where $\mathcal{L}_2G_J^n$ comes from $(\ref{eq:FLMMs})$ or $(\ref{eq:PI})$, and Lemma \ref{lemma:2.1}.
And from Lemma \ref{lemma:2.1}, it is easy to check that
\begin{equation}\label{eq:L1-2}
D_t^1 (G_J)\Big|_{t=t_n}=\left\{\begin{array}{l}(G_J^n-e^{-pU \tau}G_J^{n-1})/\tau+\mathcal{O}(\tau),\\ [5pt]
                                   (\frac{3}{2}G_J^n-2e^{-pU \tau}G_J^{n-1}+\frac{1}{2}e^{-2pU \tau}G_J^{n-2})/\tau+\mathcal{O}(\tau^2).\end{array}\right.
\end{equation}
 Now the FBDF and PI MGM fully discrete schemes can be, respectively, given as: find $G^n_J\in S_J$, such that for any $v\in S_J$ there is, respectively,
 \begin{equation}\label{eq:weak1}
 (\mathcal{L}_1G^n_J,v)+K_{\gamma,\alpha} \sum_{j=0}^n w_{n-j}^{1-\gamma}\left[B\left(G_J^j,e^{pU(j-n)\tau}v\right)-B\left(G_J^0, e^{-pU n\tau}v\right)\right]+\frac{B\left(G_J^0, e^{-pUn\tau}v\right)}{(t_n)^{1-\gamma}\Gamma(\gamma)}=\left(f^n,v\right),
 \end{equation}
 and
\begin{equation}\label{eq:weak2}
 (\mathcal{L}_1G^n_J,v)+\Gamma(\gamma)^{-1}\sum_{j=0}^{n}B\left(k^{1-\gamma}_{n-j}G_J^j,\,e^{pU(j-n)\tau}v\right)+\frac{B\left(G_J^0, e^{-pUn\tau}v\right)}{(t_n)^{1-\gamma}\Gamma(\gamma)}=\left(f^n,v\right),
 \end{equation}
where $f^n=f(x,p,t_n)$.

If replace $D^1_t(G_x)$, ${}^SD_{t}^{1-\gamma}G_x$ and $f$ in (\ref{eq:3.1}) by $\mathcal{L}_1G_x^n$, $\mathcal{L}_2G_x^n$ and $f^n$ respectively, then finding $G_x^n\in H^{\frac{\alpha}{2}}_0(\Omega)$  s.t. (\ref{eq:3.1}) holds for any $ v \in H^{\frac{\alpha}{2}}_0(\Omega)$ yields the time semi discrete scheme. Similarly one can get the space semi discrete scheme by replacing $G$ in (\ref{eq:3.1}) with $G_J (\in S_J)$ s.t. (\ref{eq:3.1}) holds for any $v\in S_J$.

In the following analysis of this subsection, we make the assumption that $pU$ is a nonnegative real number and unify the forms of (\ref{eq:weak1}) and (\ref{eq:weak2}) as:
find $G^n_J\in S_J$ such that
\begin{equation} \label{eq:fulldiscrete}
(\mathcal{L}_1G^n_J,v)+K_{\gamma,\alpha}B(\mathcal{L}_2G^n_J,v)=\left(f^n,v\right)\quad \forall v\in S_J.
\end{equation}
 For $0<\gamma<1$,  similar to the methods used in \cite{Ervin:06,Mclean:12}, after extending the real valued $k(t)$ by zero outside the interval $[0,T]$, then from Eq. (\ref{eq:Laplace}) and the Plancherel theorem,  it is easy to check that
\begin{eqnarray}\label{eq:Positive}
&& \int_0^T \left({}^SD_t^{1-\gamma} k(t)\right)\,k(t)\,dt=\frac{1}{2\pi}\int_{-\infty}^{\infty}(pU+is)^{1-\gamma}\hat{k}(is)\hat{k}(-is)ds\\\nonumber
&& \quad =\left\{\begin{array}{ll}
  \frac{1}{\pi}\int_0^{\infty}\cos\left((1-\gamma)\arctan\big(\frac{s}{pU}\big)\right)\big|pU+is\big|^{1-\gamma}\big|\hat{k}(is)\big|^2\,ds\ge0,& pU>0,\\[5pt]
  \frac{sin\frac{\pi\gamma}{2}}{\pi}\int_0^{\infty}s^{1-\gamma}\big|\hat{k}(is)\big|^2\,ds\ge 0,&pU=0.
  \end{array}\right.
\end{eqnarray}

\begin{lemma}\label{eq:lemma31}
Let $\{L_n^\gamma\}_{n=0}^{\infty}$ and $\{Q_n^\gamma\}_{n=0}^{\infty}$ be two sequences of nonnegative real numbers with $ L_n^\gamma\le Q_{n-1}^\gamma,n\ge 1$, and $\{Q_n^\gamma\}_{n=0}^{\infty}$ monotone decreasing, $L_{0}^\gamma:=Q_0^\gamma$, $Q_{-1}^\gamma:=0$.
Then for any positive integer $M$ and real vector $\left(V_0,V_1,\cdots,V_{M-1}\right)$ with $M$ real entries,
\begin{equation}\label{eq:lemma22_1}
\sum_{n=0}^{M-1}\left(\sum_{l=0}^{n-1}e^{-pU l\tau}\left(Q_l^\gamma-Q_{l-1}^\gamma\right)V_{n-l}+e^{-pU n\tau}\left(L_{n}^\gamma-Q_{n-1}^\gamma\right)V_0\right)V_n\ge 0.
\end{equation}
For $n\ge 1$, if we further require $Q_n^\gamma\le L_{n}^\gamma$, then
\begin{eqnarray}\label{eq:lemma22_2}
&&~\left(\sum_{l=0}^{n-1}e^{-pU l\tau}\left(Q_l^\gamma-Q_{l-1}^\gamma\right)V_{n-l}^\gamma+e^{-pU n\tau}\left(L_{n}^\gamma-Q_{n-1}^\gamma\right)V_0\right)V_n\\\nonumber
&&\ge \frac{1}{2}\sum_{l=0}^ne^{-pU l\tau}Q_l^\gamma V^2_{n-l}-\frac{1}{2}\sum_{l=0}^{n-1}e^{-pU l\tau}Q_l^\gamma  V^2_{n-1-l}+\frac{L_n^\gamma}{2}V_{n}^2.
\end{eqnarray}
\end{lemma}
\begin{proof}
Note  that
\begin{eqnarray*}
&& Q_0^\gamma+ \sum_{l=1}^{n-1}e^{-pU l\tau}\left(Q_l^\gamma-Q_{l-1}^\gamma\right)+e^{-pU n\tau}\left(L_{n}^\gamma-Q_{n-1}^\gamma\right)\ge L_{n}^\gamma\ge 0, \quad e^{-pU n\tau}\left(L_{n}^\gamma-Q_{n-1}^\gamma\right)<0,\\
&& Q_0^\gamma>0, \quad e^{-pU l\tau}\left(Q_l^\gamma-Q_{l-1}^\gamma\right)<0,\quad 0<e^{-pU l\tau}\le 1,\quad L_{n}^\gamma-Q_{n-1}^\gamma>Q_n^\gamma-Q_{n-1}^\gamma.
\end{eqnarray*}
Further combining with the average value inequality leads to the desired results.
\end{proof}

Eq. (\ref{eq:lemma22_1}) can be seen as the discrete analogue of (\ref{eq:Positive}). Lemma \ref{eq:lemma31} plays important role in the following proof. When using this lemma, we take
\begin{eqnarray}\label{eq:coefficient}
Q^{\gamma}_n:=
\left\{\begin{array}{ll} \tau^{-1}\pi_n^{\gamma-1}, & for \,{\rm FBDF},\\[3pt] \frac{1}{\Gamma(1-\gamma)\tau}\int_0^\tau\frac{ds}{(t_{n+1}-s)^{\gamma}},& for\, {\rm PI},\end{array}\right.\qquad n\ge 0,\quad L_0^{\gamma}:=Q_0^{\gamma},\quad L^{\gamma}_n:=\frac{t_n^{-\gamma}}{\Gamma(1-\gamma)},\quad n\ge 1,
\end{eqnarray}
where $\sum_{k=0}^{\infty}\pi^{\gamma-1}_k\xi^k=\left(\frac{1-\xi}{\tau}\right)^{\gamma-1}$.
Since $(t_{n+1}-s)^{-\gamma}<(t_n-s)^{-\gamma}$ with $s\in (0,\tau)$ and $ (t_n-s)^{-\gamma}$ with $s\in (0,\tau)$ is a convex function, then both (\ref{eq:lemma22_1}) and (\ref{eq:lemma22_2}) hold for the PI scheme. For the FBFD scheme, since  $w^{\gamma}_j\le0$, $|w^{\gamma}_{j+1}|<|w^{\gamma}_j|<w^{\gamma}_0=1/\tau^{\gamma}, j=1,2,\cdots,\sum_{j=0}^{\infty}w^{\gamma}_j=0$, and $\tau^{-1}\pi_n^{\gamma-1}=\sum_{j=0}^{n}w_j^{\gamma}$, the sequence $Q_n^\gamma$ is postive and  decreasing with $n$. Noting that \cite{Galeone:06}
 \begin{equation}\label{eq:ssss3}
\frac{Q_{n-1}^\gamma}{Q_n^\gamma}=\frac{n}{n-\gamma}>1+\frac{\gamma}{n}\ge\left(1+\frac{1}{n}\right)^\gamma=\frac{L_n^\gamma}{L_{n+1}^\gamma}
 \end{equation}
 and
  \begin{equation}\label{eq:ssss2}
 \tau^{\gamma}Q_{n-1}^{\gamma}=\frac{n^{-\gamma}}{\Gamma(1-\gamma)}-
   \frac{(\gamma+1)n^{-1-\gamma}}{2\Gamma(-\gamma)}+\mathcal{O}(n^{-2-\gamma}),
 \end{equation}
  one has $ L^{\gamma}_{n}\le Q^{\gamma}_{n-1}$. Moreover, using the log-convexity of  Gamma function on the positive reals (Bohr-Mollerup theorem), it holds that
 \begin{eqnarray}
 \Gamma(n+1-\gamma)=\Gamma\left(\gamma n+(1-\gamma)(n+1)\right)\le\Gamma(n)^{\gamma}\Gamma(n+1)^{1-\gamma}=\Gamma(n)n^{1-\gamma},
 \end{eqnarray}
 which actually means $Q^{\gamma}_n\le L^{\gamma}_{n}$. With the definitions of $Q^{\gamma}_n$ and $L^{\gamma}_{n}$ given in (\ref{eq:coefficient}), the $\mathcal{L}_2G_J^n$  defined in (\ref{eq:L11}) and (\ref{eq:L1-21}) can be uniformly rewritten as
   \begin{equation}
 \mathcal{L}_2G_J^n=\sum_{l=0}^{n-1}e^{-pU l\tau}\left(Q^{1-\gamma}_{l}-Q^{1-\gamma}_{l-1}\right)G_J^{n-l}-e^{-pU n\tau}\left(Q^{1-\gamma}_{n-1}-L^{1-\gamma}_{n}\right)G_J^0.
 \end{equation}
  
\begin{theorem} \label{eq:theorem1}
For the space semi discrete schemes corresponding to (\ref{eq:fulldiscrete}), there exists
\begin{equation}\label{theorem1_1}
\|G_J(t)\|\le \|G_J(0)\|+2\int_0^{t}\|f\|\,\mathrm{d}t;
\end{equation}
and for the two time semi discrete schemes corresponding to (\ref{eq:fulldiscrete}), one has
\begin{eqnarray}
\|G^n\|&\le &\|g\|+2K_{\gamma,\alpha}^{\frac{1}{2}}\tau^{\gamma/2}\normmm{g}+2\tau\sum_{j=1}^n\|f^n\| \quad for\,{\rm FBDF},\\
\|G^n\|&\le &4\|g\|+\frac{5K_{\gamma,\alpha}^{\frac{1}{2}}\tau^{\gamma/2}}{\sqrt{\Gamma(\gamma+1)}}\normmm{g}+18\tau\sum_{j=1}^n\|f^n\| \quad for\,{\rm PI},
\end{eqnarray}
where $G_J(0)$ can be chosen as $\Pi g$, $g$ is the initial value, $\normmm{u}^2:=B(u,u)$.
\end{theorem}
\begin{proof}
Let $\mathcal{B}_J:S_J\to S_J$ be the discrete analogue of $\nabla^{\alpha}_x$ defined by
\begin{equation}\label{eq:BB}
\left(\mathcal{B}_J \Theta,\chi\right)=B\left(\Theta,\chi\right)\quad \forall\, \Theta, \chi \in S_J.
\end{equation}
Then operator $\mathcal{B}_J$ is selfadjoint and positive define in $S_J$ w.r.t. $(\cdot,\cdot)$.
Using (\ref{eq:Positive}), from the space semi discrete scheme, one can easily derive
\begin{equation}
\left\|G_J(t_n)\right\|^2\le\left\|G_J(0)\right\|^2+2\int_0^{t_n}\left\|f\right\|\left\|G_J\right\|\,dt.
\end{equation}
 Supposing that $\left\|G_J(t^\ast)\right\|=\sup\limits_{0\le t\le t_n}\left\|G_J(t)\right\|$,  one can get (\ref{theorem1_1}) through
\begin{equation}
\left\|G_J(t_n)\right\|\le\left\|G_J(t^\ast)\right\|\le\left\|G_J(0)\right\|+2\int_0^{t_n}\left\|f\right\|dt.
\end{equation}
Next, we do the estimations for the time semi discrete schemes. Let ${\mathcal{B}}: H_0^{\frac{\alpha}{2}}(\Omega)\to H^{-\frac{\alpha}{2}}(\Omega)$ denote the usual differential operator given as
\begin{equation}
 \left\langle{\mathcal{ B}}\Theta, \chi\right\rangle=B(\Theta,\chi)\quad \forall\, \Theta,\, \chi \in H_0^{\frac{\alpha}{2}}(\Omega).
\end{equation}
It is easy to check that $\mathcal{B}$ is a boundedly invertible operator, i.e.,
\begin{equation}
C_1\|\Theta\|_{H^{\frac{\alpha}{2}}(\Omega)}\stackrel{<}{\sim}\|\mathcal{B}\Theta\|_{H^{-\frac{\alpha}{2}}}\stackrel{<}{\sim} C_2\|\Theta\|_{H^{\frac{\alpha}{2}}(\Omega)}\quad \forall \Theta \in {H_0^{\frac{\alpha}{2}}(\Omega)}.
\end{equation}
Choosing $\Psi=\left\{2^{-\frac{j\alpha }{2}}\psi_{j,k}, \,k\in \nabla_j,j\ge J_0-1\right\}$ as a Riesz basis  of $H_0^{\frac{\alpha}{2}}(\Omega)$, one can define a well-conditioned operator ${\bf B}_l: l_2(\mathcal{J})\to l_2(\mathcal{J}),\, {\bf B}_l=B(\Psi,\Psi)$ \cite{Urban:09}; for any function $ \Theta=c(\Theta)\Psi\in {H_0^{\frac{\alpha}{2}}(\Omega)},\, \chi=c(\chi)\Psi\in {H_0^{\frac{\alpha}{2}}(\Omega)}$, it follows that
\[\left({\bf B}_l c(\Theta), c(\chi)\right)_{l_2}=\left\langle\mathcal{B}\Theta,\chi\right\rangle,\]
where  $\mathcal{J}=\left\{k\in \nabla_j,j\ge J_0-1\right\}$, and $\left(\cdot,\cdot\right)_{l_2}$ denotes the inner product in the sequence space $l_2(\mathcal{J})$. Since ${\bf B}_l$ is self-adjoint positive define and boundedly invertible, there exists an unique positive square root ${\bf B}_l^{\frac{1}{2}}$ \cite{Sen:13} satisfying $\left({\bf B}_l c(\Theta), c(\chi)\right)_{l_2}=\left({\bf B}_l^{\frac{1}{2}} c(\Theta),{\bf B}_l^{\frac{1}{2}} c(\chi)\right)_{l_2}$.

Choosing $v=G^n$ in  the time semi discrete schemes, for $ n\ge 1$, one gets
\begin{eqnarray}
\left(\mathcal{L}_1G^n, G^n\right)&=& K_{\gamma,\alpha}\sum_{l=0}^{n-1}e^{-pU l\tau}\left(Q^{1-\gamma}_{l-1}-Q^{1-\gamma}_l\right)B(G^{n-l},G^n)\\\nonumber
&&+K_{\gamma,\alpha}e^{-pU n\tau}\left(Q^{1-\gamma}_{n-1}-L^{1-\gamma}_{n}\right)B(G^0,G^n)+\left(f^n,G^n\right).
\end{eqnarray}
Replacing $n$ by $j$ and then adding $j$ from $1$ to $n$, one has
\begin{eqnarray}
\sum_{j=1}^n\left(\mathcal{L}_1G^j, G^j\right)&=& K_{\gamma,\alpha}\sum_{j=0}^n\sum_{l=0}^{j-1}e^{-pU l\tau}\left(Q^{1-\gamma}_{l-1}-Q^{1-\gamma}_l\right)B\left(G^{j-l},G^j\right) \\
 &&+K_{\gamma,\alpha}\sum_{j=0}^ne^{-pU j\tau}\left(Q^{1-\gamma}_{j-1}-L^{1-\gamma}_{j}\right)B\left(G^0,G^j\right)+K_{\gamma,\alpha}Q_0^{1-\gamma}B\left(G^0,G^0\right)+\sum_{j=1}^n\left(f^j,G^j\right).\nonumber
\end{eqnarray}
 Using  $B\left(G^{j-l},G^j\right)=\left({\bf B}_l^{\frac{1}{2}} c(G^{j-l}),{\bf B}_l^{\frac{1}{2}}c(G^j)\right)_{l_2}$ and Lemma \ref{eq:lemma31}, one obtains
\begin{equation}
\sum_{j=1}^n\left(\mathcal{L}_1G^j, G^j\right)\le K_{\gamma,\alpha} Q_0^{1-\gamma}B(G^0,G^0)+\sum_{j=1}^n\left(f^j,G^j\right).
\end{equation}
Supposing $\left\|G^m\right\|=\max\limits_{0\le j\le n}\left\|G^j\right\|$, the FBDF scheme yields
\begin{equation}
\left\|G^m\right\|^2\le\left(2\tau\sum_{j=1}^n\left\|f^j\right\|\right)\left\|G^m\right\|+\left\|G^0\right\|^2+2K_{\gamma,\alpha}\tau Q_0^{1-\gamma}\normmm{G^0}^2.\\
\end{equation}
For the PI scheme,  when $ n\ge 2$, it is easy to check that
\begin{eqnarray}
\tau \left(\mathcal{L}_1G_J^n, G_J^n\right)\ge\frac{1}{4}\left\|G_J^n\right\|^2-\frac{1}{4}\left\|G_J^{n-1}\right\|^2
+\frac{e^{-2pU \tau}}{4}\left\|2e^{pU\tau}G_J^{n}-G_J^{n-1}\right\|^2-\frac{e^{-2pU \tau}}{4}\left\|2e^{pU \tau}G_J^{n-1}-G_J^{n-2}\right\|^2,\nonumber
\end{eqnarray}
whence
\begin{equation}\label{eq:relationeq}
\tau\sum_{j=2}^n\left(\mathcal{L}_1G^j, G^j\right)\ge\frac{1}{4}\left\|G_J^n\right\|^2-\frac{9}{4}\left\|G_J^1\right\|^2-\frac{1}{2}\left\|G_J^0\right\|^2.
\end{equation}
So it follows that
\begin{equation}
\left\|G^m\right\|^2\le\left(18\tau\sum_{j=1}^n\left\|f^n\right\|\right)\left\|G^m\right\|+18K_{\gamma,\alpha}\tau Q_0^{1-\gamma}\normmm{G^0}^2+11\left\|G^0\right\|^2,\quad n\ge1. \end{equation}
Combining with that  $ a, b,c>0$ and $a^2\le a\cdot b+c$ lead to $ a\le b+\sqrt{c}$, the desired results are obtained.
\end{proof}
The obtained stable space semi discrete schemes can also be solved by the numerical inverse Laplace transform or matrix function technique. And the stable time semi discrete schemes can also be solved by spectral or finite element methods.


\begin{theorem}\label{theorem:11}
The {\rm MGM} fully discrete schemes (\ref{eq:fulldiscrete}) are unconditionally stable.  Let $e^n=G_x(p,t_n)-G^n_J$. Then for the sufficiently regular solution $G_x(p,t)$, the error estimate of the {\rm PI} ({\rm or  FBDF}) scheme (\ref{eq:fulldiscrete}) is 
\begin{equation}\label{eq:err1-1}
\left\|e^n\right\|^2+ 2K_{\gamma,\alpha}\tau^{\gamma}/\Gamma(1+\gamma)\normmm{e^n}^2\le C\left(2^{J(\alpha-2d)}+\tau^{\gamma}2^{J(\alpha-2d)}+\tau^{2+2\gamma}\,\left( {\rm or}~\tau^2\right)\right).
\end{equation}
 In addition, if $\left\|\Lambda G -G\right\|\stackrel{<}{\sim}2^{-Jd}\left\|G\right\|_{H^d(\Omega)}$ holds (it is obvious when  $\alpha=2$), then the convergence order related to space can be improved to $\mathcal{O}\left(2^{-2Jd}+\tau^{\gamma}2^{J(\alpha-2d)}\right)$. Here $d$ is the order of B-spline.
\end{theorem}
\begin{proof}
Taking $v=G^n_J$ in (\ref{eq:fulldiscrete}) leads to
\begin{eqnarray}
 \left(\mathcal{L}_1G_J^n, G_J^n\right)&=&K_{\gamma,\alpha}\sum_{l=0}^{n-1}e^{-pU l\tau}\left(Q^{1-\gamma}_{l-1}-Q^{1-\gamma}_l\right)B(G_J^{n-l},G_J^n) \\\nonumber
 &&+K_{\gamma,\alpha}e^{-pU n\tau}\left(Q^{1-\gamma}_{n-1}-L^{1-\gamma}_{n}\right)B(G_J^0,G_J^n)+\left(f^n,G_J^n\right).\nonumber
\end{eqnarray}
Noting that $B\left(G_J^{n-l},G_J^n\right)\le \normmm{G_J^{n-l}}\normmm{G_J^n}$ and using Lemma \ref{eq:lemma31}, one obtains
\begin{eqnarray} \label{eq:fullstable}
 2\left(\mathcal{L}_1G_J^n, G_J^n\right)&\le &K_{\gamma,\alpha}\sum_{k=0}^{n-1}e^{-pU k\tau}Q^{1-\gamma}_k\normmm{G_J^{n-1-k}}^2 \\[3pt]\nonumber
 &&-K_{\gamma,\alpha}\sum_{k=0}^ne^{-pU k\tau}Q^{1-\gamma}_k\normmm{G_J^{n-k}}^2-\,K_{\gamma,\alpha}L^{1-\gamma}_n\normmm{G_J^{n}}^2+2\left(f^n,G_J^n\right).\nonumber
\end{eqnarray}

Let us first consider  the FBDF scheme.
Replacing $n$ by $j$ and then summing from $1$ to $n$, it yields that
\begin{eqnarray}
\left\|G^n_J\right\|^2&+&K_{\gamma,\alpha}\tau\sum_{k=0}^ne^{-pU k\tau}Q^{1-\gamma}_k\normmm{G_J^{n-k}}^2\\ \nonumber
&\le& \left\|G^0_J\right\|^2+K_{\gamma,\alpha}\tau Q^{1-\gamma}_0\normmm{G_J^0}^2-K_{\gamma,\alpha}\tau\sum_{j=1}^nL^{1-\gamma}_j\normmm{ G_J^j}^2+2\tau \sum_{j=1}^n\left(f^j,G_J^j\right).\nonumber
\end{eqnarray}
Using  the fractional Poincar\'e-Friedrich's \cite{Ervin:06} and Yong's inequalities, one has
\begin{eqnarray}
&&-K_{\gamma,\alpha}\tau\sum_{j=1}^nL^{1-\gamma}_j\normmm{G_J^j}^2+2\tau \sum_{j=1}^n\left(f^j,G_J^j\right)\\\nonumber
&&\le -\sum_{j=1}^nK_{\gamma,\alpha}\tau t_j^{\gamma-1}/\left(C_f\Gamma(\gamma)\right)\left\|G_J^j\right\|^2+2\tau \sum_{j=1}^n\left(f^j,G_J^j\right)\le \tau{4\Gamma(\gamma)C_f}/\left({K_{\gamma,\alpha}t_n^{\gamma-1}}\right)\sum_{j=1}^n\left\|f^j\right\|^2,\nonumber
\end{eqnarray}
where $C_f$ comes from $\left\|G_J^j\right\|\le C_f\normmm{G_J^j}$. Then  the full discrete FBDF scheme (\ref{eq:fulldiscrete}) is stable, and
\begin{eqnarray}\label{eq:FBDFstable}
\left\|G^n_J\right\|^2&+&K_{\gamma,\alpha}\tau\sum_{k=0}^ne^{-pU k\tau}Q^{1-\gamma}_k\normmm{G_J^{n-k}}^2\\
&\le& \left\|G^0_J\right\|^2+K_{\gamma,\alpha}\tau Q^{1-\gamma}_0\normmm{G_J^0}^2+ \tau{4\Gamma(\gamma)C_f}/\left({K_{\gamma,\alpha}t_n^{\gamma-1}}\right)\sum_{j=1}^n\left\|f^j\right\|^2.\nonumber
\end{eqnarray}

Next, we consider the PI scheme.  For $n\ge 2$, starting from (\ref{eq:fullstable}) and (\ref{eq:relationeq}), one obtains
\begin{eqnarray}\label{eq:PI1}
\left\|G^n_J\right\|^2&+&2K_{\gamma,\alpha}\tau\sum_{k=0}^ne^{-pU k\tau}Q^{1-\gamma}_k\normmm{G_J^{n-k}}^2\\
 &\le& 2\left\|G^0_J\right\|^2+9\|G_J^1\|^2+2K_{\gamma,\alpha}\tau\sum_{k=0}^1e^{-pU k\tau}Q^{1-\gamma}_k\normmm{G_J^{1-k}}^2+ \tau {8\Gamma(\gamma)C_f}/\left({K_{\gamma,\alpha}t_n^{\gamma-1}}\right)\sum_{j=2}^n\left\|f^j\right\|^2.\nonumber
\end{eqnarray}
And when $n=1$, applying Young's inequality to the last term of (\ref{eq:fullstable}) results in
\begin{eqnarray}\label{eq:PI2}
\left\|G^1_J\right\|^2+2K_{\gamma,\alpha}\tau\sum_{k=0}^1e^{-pU k\tau}Q^{1-\gamma}_k\normmm{G_J^{1-k}}^2 \le2\left\|G^0_J\right\|^2+2K_{\gamma,\alpha}\tau Q^{1-\gamma}_0\normmm{G_J^0}^2+4\tau^2\left\|f^1\right\|^2.
\end{eqnarray}
 Now, we derive the convergence results. Let
\begin{eqnarray*}
&& e^n=G(t_n)-G^n_J=\big(G(t_n)-\Lambda G(t_n)\big)+\big(\Lambda G(t_n)-G^n_J\big)=\theta^n+\eta^n,\\[5pt]
&&\mathcal{R}^n=\left(\mathcal{L}_1G^n-\frac{\partial G(t_n)}{\partial t}-pU G(t_n)\right)+K_{\gamma,\alpha}\left(\mathcal{L}_2(\nabla_{x}^{\alpha} G^n)-{}^SD_t^{1-\gamma}\,\nabla_{x}^{\alpha} G^n\right)-\mathcal{L}_1\theta^n,
\end{eqnarray*}
where $G(t_n)=G_x\left(p,t_n\right)$. Combining (\ref{eq:FLMMs}) and (\ref{eq:FBDFstable}) results in
\begin{eqnarray}
\left\|\eta^n\right\|^2+K_{\gamma}\tau Q^{1-\gamma}_0\normmm{\eta^n}^2\stackrel {<}{\sim}\left\|\eta^0\right\|^2+\tau Q^{1-\gamma}_0\normmm{\eta^0}^2+ \tau\sum_{j=1}^n\left\|\mathcal{R}^j\right\|^2,
\end{eqnarray}
where $\normmm{\eta^0}\le\normmm{G(0)-G_J^0}+\normmm{\theta^0}$ and
\begin{eqnarray}
\tau\sum_{j=1}^n\left\|\mathcal{R}^j\right\|^2\stackrel{<}{\sim}
\sum_{j=1}^n\int_{t_{j-1}}^{t_j}\left\|\left(\Lambda-I\right)D_t^1G\right\|^2\,ds+\tau^2\int_0^{t_n}\left\|D_t^2G\right\|^2\,ds+C\tau^2.
\end{eqnarray}
Combining (\ref{eq:PIError}), (\ref{eq:PI1}), and (\ref{eq:PI2}) leads to
\begin{eqnarray}
 \left\|\eta^n\right\|^2+2K_{\gamma}\tau Q^{1-\gamma}_0\normmm{\eta^{n}}^2 \stackrel{<}{\sim} \left\|\eta^0\right\|^2+\tau Q^{1-\gamma}_0\normmm{\eta^0}^2+\tau^2\left\|\mathcal{R}^1\right\|^2+\tau\sum_{j=2}^n\left\|\mathcal{R}^j\right\|^2,\quad n\ge 1,
\end{eqnarray}
where
\begin{eqnarray*}
&&\tau^2\|\mathcal{R}^1\|^2\stackrel {<}{\sim}\tau\int_{0}^{\tau}\left\|\left(\Lambda-I\right)D_t^1G\right\|^2\,ds+\tau^3\int_0^{\tau}\left\|D_t^2G\right\|^2\,ds+C\tau^{3+2\gamma},\\
&&\tau\sum_{j=2}^n\|\mathcal{R}^j\|^2\stackrel{<}{\sim}\int_{0}^{t_n}\left\|\left(\Lambda-I\right)D_t^1G\right\|^2\,ds+\tau^4\sum_{j=2}^n\int_{t_{j-2}}^{t_j}\left\|D_t^3G\right\|^2dt+C\tau^{2+2\gamma}.
\end{eqnarray*}
Noting that
 \begin{equation} \label{eq:projectsss}
\left\|\Lambda G -G\right\|\stackrel{<}{\sim}\normmm{\Lambda G-G}\sim \left\|\Lambda G -G\right\|_{H^{\frac{\alpha}{2}}(\Omega)},
\end{equation}
and taking $G_J^0$ as one of $ \Lambda g$ and $\Pi g$, the desired results are obtained.
\end{proof}

\subsection{Numerical schemes $\mathrm{II}$}
Taking the Laplace transform on both sides of (\ref{eq:model1}) leads to
\begin{equation}\label{eq:Laplacequation}
\left(s+pU\right)\hat{G}-K_{\gamma,\alpha}\left(s+pU\right)^{1-\gamma}\nabla_x^\alpha \hat{G}=g.
\end{equation}
Then
\begin{equation}\label{eq:Laplacequation}
\left(s+pU\right)^\gamma\hat{G}-\left(s+pU\right)^{\gamma-1}g-K_{\gamma,\alpha} \nabla_x^\alpha \hat{G}=0.
\end{equation}
Taking the inverse Laplace transform to the above equation results in
\begin{equation}\label{eq:nonsym}
{}^SD_{*,t}^{\gamma}G_x\left(p,t\right)-K_{\gamma,\alpha}\nabla_x^\alpha G_{x}\left(p,t\right)=0,
\end{equation}
with the initial condition $G_x(p,0)=g(x,p)$. For making a more general discussion, we add a force term $f(x,p,t)$ on the right hand of (\ref{eq:nonsym}). And the corresponding space semi discrete scheme reads:  find $G_J \in S_J$ such that
\begin{equation}\label{eq:semidiscret41}
\left({}^SD_{*,t}^{\gamma}G_J, \,v\right)+K_{\gamma,\alpha}B(G_J,\,v)=(f,\,v) \quad \forall v\in S_J.
\end{equation}
\begin{lemma} [see \cite{Vergara:08}]\label{eq:Lemma4.1}
Let $k(t)=t^{-\gamma}/\Gamma(1-\gamma),\, T>0$, and $\mathcal{H}(\Omega)$ a real Hilbert space. Suppose that $v(\cdot,t)\in H^1\left([0,T],\mathcal{H}\right)$ and $v(\cdot,0)\in \mathcal{H}$. Then
\begin{equation}
\left\langle v(t),\,\frac{\partial}{\partial t}(k*v)(t)\right\rangle_{\mathcal{H}}\ge\frac{1}{2}\frac{d}{d t}\left(k*\left\|v\right\|^2_{\mathcal{H}}\right)(t)+\frac{1}{2}k(t)\left\|v(t)\right\|_{\mathcal{H}}^2,\quad a.e.~~ t\in(0,T).
\end{equation}
\end{lemma}
\begin{theorem}
For the space semi discrete scheme (\ref{eq:semidiscret41}) with $0<\gamma<1$, it holds that
\begin{equation}
\int_0^t\left(\left\|G_J\right\|^2+K_{\gamma,\alpha}B\left(G_J,\,G_J\right)\right)\,ds\le
 C \left(\left\|g\right\|^2+ \int_0^t\left\|f\right\|^2\,ds\right).
\end{equation}
\end{theorem}
\begin{proof}
Choosing  $v=G_J$ in (\ref{eq:semidiscret41}), we obtain
\begin{equation}
\left({}^SD_{t}^{\gamma}\left[G_J-e^{-pU t}g\right], \,G_J\right)+K_{\gamma,\alpha}B\left(G_J,\,G_J\right)=\left(f,\,G_J\right),
\end{equation}
which can be rewritten as
\begin{equation}
\qquad\left(e^{pU t}G_J,\,\frac{\partial}{\partial t}\big[k*(e^{pU t}G_J)\big](t)\right)+K_{\gamma,\alpha}e^{2pU t}B\left(G_J,\,G_J\right)=\frac{e^{pU t}t^{-\gamma}}{\Gamma\left(1-\gamma\right)}\left(g,\,G_J\right)+e^{2pU t}\left(f,\,G_J\right).
\end{equation}
Using Lemma \ref{eq:Lemma4.1}, one finds
\begin{eqnarray}
\frac{1}{2}e^{-2pUt}\frac{d}{d t}\left(k*\left(e^{2pU t}\left\|G_J\right\|^2\right)\right)(t)+\frac{1}{2}k(t)\left\|G_J\right\|^2+K_{\gamma,\alpha}B\left(G_J,\,G_J\right)\le e^{-pU t}k(t)\left(g,\,G_J\right)+\left(f,\,G_J\right).\nonumber
\end{eqnarray}
By Young's inequality and (\ref{eq:coercive}), it  yields  that
\begin{equation}\label{eq:ss1}
{}_0D_t^{\gamma,2pU}\left\|G_J\right\|^2+K_{\gamma,\alpha}B\left(G_J,\,G_J\right)\le k(t)e^{-2pU t}\left\|g\right\|^2+{1}/\left({K_{\gamma,\alpha}C_2}\right)\left\|f\right\|^2.
\end{equation}
Performing the operator ${}_0D_t^{-\gamma,2pU}$ on both sides of (\ref{eq:ss1}) and using Lemma \ref {lemma:2.2}, one obtains
\begin{equation} \label{eq:ss2}
\left\|G_J\right\|^2+K_{\gamma,\alpha}\,{}_0D_t^{-\gamma,2pU}B\left(G_J,\,G_J\right)\le e^{-2pU t}\left\|g\right\|^2+ {1}/\left({K_{\gamma,\alpha}C_2}\right){}_0D_t^{-\gamma,2pU}\left\|f\right\|^2.
\end{equation}
Adding (\ref{eq:ss1}) to (\ref{eq:ss2}) and integrating from $0$ to $t$, one reaches the conclusion.
\end{proof}

Now, thanks  to (\ref{eq:FLMMs}) and (\ref{eq:PI}), we denote
\begin{equation}\label{eq:discrete}
\mathcal{L}G^n_J=\left\{\begin{array}{ll}e^{-pU n\tau}\sum_{j=0}^n w_{n-j}^{\gamma}\left(e^{pU j\tau}G_J^j-G_J^0\right),&for\,{\rm FBDF},\\[3pt]\Gamma(1-\gamma)^{-1}\sum_{j=0}^{n}e^{pU(j-n)\tau}k^{\gamma}_{n-j}G_J^j,& for\,{\rm PI}.\end{array}\right.
\end{equation}
Then the fully discrete scheme of (\ref{eq:semidiscret41}) is: find $G_J^n\in S_J$ such that
\begin{equation} \label{eq:sss5}
(\mathcal{L}G^n_J,\,v)+K_{\gamma,\alpha}B(G_J^n,\,v)=(f,\,v)\qquad \forall v\in S_J.
\end{equation}
In the following analysis of this subsection, we suppose that $pU(x)\ge0$ (for the more general case, see Example 4.3). With the definition of $Q_n^{\gamma}, n=-1,0,1,\cdots$  given in (\ref{eq:coefficient}), Eq. (\ref{eq:discrete}) can also be uniformly rewritten as  \begin{equation}
 \mathcal{L}G_J^n=\sum_{l=0}^{n-1}e^{-pUl\tau}\left(Q_l^{\gamma}-Q_{l-1}^{\gamma}\right)G_J^{n-l}-e^{-pUn\tau}Q_{n-1}^{\gamma}G_J^0.
 \end{equation}
\begin{lemma} \label{eq:lemmass1}
For any positive integer $n$ and real vector $\{V_0, V_1,\cdots, V_n\}$ with $n+1$ real entries, one has
\begin{equation}\label{eq:sss3}
 V_n\cdot\mathcal{L}V_n\ge\left(\frac{Q_0^{\gamma}}{2}+\frac{Q^{\gamma}_{n-1}}{4}\right)V_n^2+ \frac{1}{2}\sum_{l=1}^{n-1} e^{-2pU l\tau}\left(Q^{\gamma}_l-Q^{\gamma}_{l-1}\right)V_{n-l}^2-e^{-2pU n\tau}Q^{\gamma}_{n-1}V_0^2.
\end{equation}
\end{lemma}
\begin{proof}
Using  Young's inequality, one has
\begin{eqnarray}
 &&V_n\cdot\mathcal{L}V_n=e^{-2pU n\tau}\left(\sum_{l=0}^{n-1}e^{pU (n-l)\tau}\left(Q^{\gamma}_l-Q^{\gamma}_{l-1}\right)V_{n-l}-Q^{\gamma}_{n-1}V_0\right)\left(e^{pU n\tau}V_n\right)\\\nonumber
 &&\ge  Q_0^{\gamma}V_n^2+e^{-2pU n\tau}\sum_{l=1}^{n-1}\left(Q^{\gamma}_l-Q^{\gamma}_{l-1}\right)\frac{e^{2pU (n-l)\tau}V^2_{n-l}+e^{2pU n\tau}V^2_n}{2}
-e^{-2pU n\tau}Q^{\gamma}_{n-1}\left(V_0^2+\frac{e^{2pU n\tau}V^2_n}{4}\right),\\\nonumber
\end{eqnarray}
which arrives at the conclusion.
\end{proof}

In fact, when $pU$ is constant, from Lemma \ref{eq:Lemma4.1} one has
\begin{eqnarray}\label{eq:ssss1}
\left({}^SD_{*,t}^{\gamma}G, \,G\right)&=&e^{-2pU t}\left(e^{pU t}G_J,\,\frac{\partial}{\partial t}\big[k*(e^{pU t}G)\big](t)\right)-e^{-pU t}k(t)(g,\,G)\\\nonumber
& \ge&\frac{1}{2}{}_0D_{t}^{\gamma,2pU}\left\|G\right\|^2+\frac{1}{2}k(t)\left\|G(t)\right\|^2-k(t)\left(e^{-2pU t}\left\|g\right\|^2+\frac{\left\|G(t)\right\|^2}{4}\right)\\\nonumber
&=& \frac{1}{2}{}_0D_{t}^{\gamma,2pU}\left\|G\right\|^2+\frac{k(t)}{4}\left\|G(t)\right\|^2-e^{-2pU t}k(t)\left\|g\right\|^2,
\end{eqnarray}
and there exists $\lim\limits_{n\to\infty}Q^{\gamma}_{n-1}\to k(t_n)$ because of (\ref{eq:ssss2}). Then (\ref{eq:sss3}) can be regarded as a discrete analogue of (\ref{eq:ssss1}). 

\begin{theorem}\label{theorem:22}
The {\rm MGM} fully discrete schemes (\ref {eq:sss5}) are unconditionally stable.  Let $e^n=G_x(p,t_n)-G^n_J$. Then for the sufficiently regular solution $G_x(p,t)$, the error estimate of the {\rm PI} ({\rm or  FBDF}) scheme (\ref {eq:sss5}) is 
\begin{equation} \label{eq:err1-2}
\left\|e^n\right\|^2+ 2K_{\gamma,\alpha}\tau^{\gamma}\Gamma(2-\gamma)B\left(e^n,e^n\right)\le C\left(2^{J(\alpha-2d)}+\tau^{\gamma}2^{J(\alpha-2d)}+\tau^{4-2\gamma}\,\left({\rm or~}\tau^2\right)\right).
\end{equation}
 In addition, if $\left\|\Lambda G -G\right\|\stackrel{<}{\sim}2^{-Jd}\left\|G\right\|_{H^d(\Omega)}$ holds (it is obvious when  $\alpha=2$), then the convergence order related to space can be improved to $\mathcal{O}\left(2^{-2Jd}+\tau^{\gamma}2^{J(\alpha-2d)}\right)$. Here $d$ is the order of B-spline.
\end{theorem}

\begin{proof}
Choosing $v=G^n_J$ in (\ref{eq:sss5}),  using Lemma \ref{eq:lemmass1} and Young's inequality one obtains
\begin{eqnarray}
 &&\left(\frac{Q_0^{\gamma}}{2}+\frac{Q^{\gamma}_{n-1}}{4}\right)\left\|G_J^n\right\|^2+K_{\gamma}B(G_J^n,\,G_J^n)\\\nonumber
&&\le\frac{1}{2}\sum_{l=1}^{n-1} \left(Q^{\gamma}_{l-1}-Q^{\gamma}_l\right)\left\|e^{-pU l\tau}G_J^{n-l}\right\|^2+Q^{\gamma}_{n-1}\left\|e^{-pU n\tau}G_J^0\right\|^2+\frac{Q^{\gamma}_{n-1}}{4}\left\|G_J^n\right\|^2+\Gamma(1-\gamma)t_n^{\gamma}\|f^n\|^2,
\end{eqnarray}
where $Q_{n-1}^{\gamma}\ge L_n^{\gamma}=\frac{t_n^{-\gamma}}{\Gamma(1-\gamma)}$ is used. Then
\begin{eqnarray}
 && Q_0^{\gamma}\left\|G_J^n\right\|^2+2K_{\gamma}B(G_J^n,\,G_J^n)\\\nonumber
&&\quad\le \sum_{l=1}^{n-1} \left(Q^{\gamma}_{l-1}-Q^{\gamma}_l\right)\left\|e^{-pU l\tau}G_J^{n-l}\right\|^2+Q^{\gamma}_{n-1}\left(2\left\|e^{-pU n\tau}G_J^0\right\|^2+2\Gamma^2(1-\gamma)t_n^{2\gamma}\|f^n\|^2\right).
\end{eqnarray}
Using mathematical induction and $\sum_{l=1}^{j-1} \left(Q^{\gamma}_{l-1}-Q^{\gamma}_l\right)+Q^{\gamma}_{j-1}=Q_0^{\gamma}$, one obtains
\begin{equation}
\left\| G_J^n\right\|^2+2K_{\gamma,\alpha}/Q_0^{\gamma}B\left(G_J^n,G_J^n\right)\le2\left\|G_J^0\right\|^2+2\Gamma^2(1-\gamma)T^{2\gamma}\max\limits_n\left\|f^n\right\|^2,
\end{equation}
for $n=1,2,\cdots,N$.

Now, we consider the convergence results. Let
\begin{eqnarray*}
&& e^n=G(t_n)-G^n_J=\big(G(t_n)-\Lambda G(t_n)\big)+\big(\Lambda G(t_n)-G^n_J\big)=\theta^n+\eta^n,\\[3pt]
&&\mathcal{R}^n=\left(\mathcal{L}G^n-{}^SD_{*,t_n}^{\gamma}G\right)-\mathcal{L}\left(I-\Lambda\right)G^n=\mathcal{R}_1^n-\mathcal{R}_2^n,
\end{eqnarray*}
where $G(t_n)=G_x\left(p,t_n\right)$. Using the property $B(\theta^n,v)=0\,\,\forall v\in S_J$ at once results in
\begin{equation}
\left\| \eta^n\right\|^2+2K_{\gamma,\alpha}/Q_0^{\gamma}B\left(\eta^n,\eta^n\right)\le2\left\|\eta^0\right\|^2+C\left(\max\limits_n\|\mathcal{R}_1^n\|^2+\max\limits_n\left\|\mathcal{R}_2^n\right\|^2\right).
\end{equation}
And it holds that
\begin{eqnarray}
\left\|\mathcal{R}_2^n\right\|&=&\left\|e^{-pUt_n}\sum_{l=0}^{n-1}Q^{\gamma}_{n-1-l}\left(e^{pU t_{l+1}}\theta^{l+1}- e^{pU t_l}\theta^{l}\right)\right\|\\\nonumber
&\le&\sum_{l=0}^{n-1}Q^{\gamma}_{n-1-l} \int_{t_l}^{t_{l+1}}\left\|D_t^1\theta\right\|dt\le C\sum_{l=0}^{n-1}\tau Q^{\gamma}_{n-1-l}\max\limits_{0\le t\le T}\left\|D_t^1\theta\right\|.
\end{eqnarray}
Note that
\begin{equation}
 \sum_{l=0}^{n-1}(-1)^l{\gamma-1 \choose l}=(-1)^{n-1}{\gamma-2\choose n-1}=\frac{n^{1-\gamma}}{\Gamma(2-\gamma)}+\mathcal{O}(n^{-\gamma}),
\end{equation}
where the Stirling formula is used. Hence
\begin{equation}
\sum_{l=0}^{n-1}\tau Q^{\gamma}_{n-1-l}=
\left\{\begin{array}{ll}\frac{t_n^{1-\gamma}}{\Gamma(2-\gamma)}+\tau \mathcal{O}(t_n^{-\gamma}),&for\,{\rm FBDF},\\[5pt]{}_0D_{t_n}^{-(1-\gamma),0}1=
\frac{t_n^{1-\gamma}}{\Gamma(2-\gamma)},& for\,{\rm PI}.
\end{array}\right.
\end{equation}
Together with (\ref{eq:projectsss}) and the triangle inequality, the desired estimate is obtained.
\end{proof}

\section{Algorithm Implementations and Numerical Results}
Now, we focus on the implementations of the algorithms presented in the above sections. Four numerical examples are simulated. The first one is to confirm the theoretical results and the effectiveness of the proposed schemes. The second one is to describe the implementation of the wavelet precondition and show its acceleration effects. And the third and fourth ones are used to show how to solve the more general Feynman-Kac equations.
In the following Tables, `Err-2' denotes the $L_2$ error; `Err-1' is the square root of the left hand side of (\ref{eq:err1-1}) or (\ref{eq:err1-2}); and `N' denotes the number of points used in time direction.

To implement the algorithms above, the key is to firstly dig out the structure of the stiffness matrix, then find the
efficient way of generating and storing the matrix, and solving the resulting algebra system.  The following observation can be used to reduce the computation and storage cost of generating the differential matrix of the MGM schemes.
Let
\begin{eqnarray}
 &&~\phi^2(x):=x_{+}-2(x-1)_{+}+(x-2)_{+};\\[5pt] \label{eq:spline2}
&&\left\{\begin{array}{l}
\phi^{3}(x):=\frac{1}{2}x_{+}^2-\frac{3}{2}(x-1)_{+}^2+\frac{3}{2}(x-2)_{+}^2-\frac{1}{2}(x-3)_{+}^2,\\[5pt]
\phi^{3,0}(x):=-\frac{3}{2}x_{+}^2+2x_{+}+2(x-1)_{+}^2-\frac{1}{2}(x-2)_{+}^2,
\end{array}\right.\label{eq:spline3}
\end{eqnarray}
which have the well property of symmetry $\phi^d(x)=\phi^d(d-x)$ for $d=2,\,3$. Defining
\begin{eqnarray}
&&\qquad\qquad\qquad\Phi_j^{2}(x):=\left\{\phi^2_{j,k}(x):=2^{j/2}\phi^2(2^jx-k)\big|_{k=0}^{2^j-2}\right\},\\
&&\Phi_j^{3}(x):=\left\{2^{j/2}\phi^{3,0}\left(2^jx\right),\,\phi^3_{j,k}(x):=2^{\frac{j}{2}}\phi^3\left(2^jx-k\right)\left|_{k=0}^{2^j-3}\right.,\,2^{j/2}\phi^{3,0}\left(2^j(1-x)\right)\right\},\nonumber
\end{eqnarray}
then $S_j^{2}={\rm span}\left\{\Phi_j^{2}(x)\right\}$ and $S_j^{3}={\rm span}\left\{\Phi_j^{3}(x)\right\}, j=J_0,J_0+1,\cdots,$  constitute the MRA of $L^2(0,1)$ for $d=2$ and $ d=3$, respectively, with the regularity  $S_j^{d}\subset H^{s}(0,1),s< d-\frac{1}{2}$.
Now, we consider the calculation of the left Riemann-Liouville differential matrix
\begin{equation}
A^d_l:=\left({}_0D_x^{\frac{\alpha}{2},0}\Phi_J^d,{}_xD_1^{\frac{\alpha}{2},0}\Phi_J^d\right)=-\left({}_0D_x^{\alpha-1,0}\Phi_J^d,\frac{d}{dx}\Phi_J^d\right),\quad d=2,\,3.
\end{equation}
Obviously, its transpose is the right Riemann-Liouville differential matrix $A_r^d:=\left({}_xD_1^{\frac{\alpha}{2},0}\Phi_J^d,{}_0D_x^{\frac{\alpha}{2},0}\Phi_J^d\right)$.
One can see that the bases of $S^{2}_J$ are obtained by dilating and translating of a single function $\phi^2(x)$, being the same  for  $S_J^3$ except the one at the boundaries, these features allow us to use the following fact:
\begin{eqnarray*}
  &&\left({}_0D_x^{\alpha-1,0}\phi_{J,k_1}^d(x),\frac{d}{dx}\phi_{J,k2}^d(x)\right)\\
 &&\quad=\frac{2^{J}}{\Gamma(2-\alpha)}\int_{2^{-J}k_2}^{2^{-J}(d+k_2)}\int_0^x \left(x-s\right)^{1-\alpha}\mathrm{d}\phi^{d}\left(2^Js-k_1\right) \,\mathrm{d} \phi^{d}\left(2^Jx-k_2\right)\\
 &&\quad=\frac{2^J}{\Gamma(2-\alpha)}\int_{0}^{d}\int_0^{2^{-J}(x+k_2)} \left(2^{-J}\left(k_2+x\right)-s\right)^{1-\alpha} \mathrm{d}\phi^{d}\left(2^Js-k_1\right)\, \mathrm{d}\phi^{d}\\
 &&\quad=\frac{2^{2J\alpha}}{\Gamma(2-\alpha)}\int_{0}^{d}\int_{-k_1}^{x+k_2-k_1}(k_2+x-s-k_1)^{1-\alpha} \mathrm{d}\phi^{d} \mathrm{d}\phi^{d}\\
 &&\quad=\frac{2^{2J\alpha}}{\Gamma(2-\alpha)}\int_{0}^{d}\int_{0}^{x+k_2-k_1}(x-s+k_2-k_1)^{1-\alpha} \mathrm{d}\phi^{d} \mathrm{d}\phi^{d},
 \end{eqnarray*}
 whose result just depends on the value $k_2-k_1$.
Hence the matrix $A^d_l$ has a quasi-Toeplitz structure, viz., for $d=2$, it is a Toeplitz matrix; and for $d=3$, using the symmetry of $\phi^3(x)$ yields a  structure of the form:
\begin{eqnarray}
&& \left(\begin{array}{ccc}a_1&r(\boldsymbol {\mathrm{a_2}})^T&0\\ \boldsymbol {\mathrm{a_1}}&M_{(2^J-2)\times (2^J-2)}&\boldsymbol {\mathrm{a_2}}\\a_2&r(\boldsymbol {\mathrm{a_1}})^T&a_1\end{array}\right)_{2^J\times 2^J},
\end{eqnarray}
where $a_i$ are real numbers; $\boldsymbol{\mathrm{a_i}}$ are vectors, $r(\boldsymbol{\mathrm{a_i}})$  the reverse order of $\boldsymbol{\mathrm{a_i}}$; and $M_{N\times N}$ is a Toeplitz matrix.

Therefore, together with the property of compact support of the bases, one only needs to compute $2^J$ and $2(2^J+1)$ elements for $d=2$ and $d=3$, respectively, rather than the full matrix obtained by the other variational methods.  And the special structure of the bases allows us to achieve even more. Letting $H(x)$ be the Heaviside function, it formally holds that
 \begin{equation}
 {}_0 D_x^{\alpha-1}\left(H(x-x_0)v(x)\right)=H(x-x_0){}_{x_0}D_x^{\alpha-1}v(x), \quad  x_0\ge 0.
 \end{equation}
Combining with the well-known formulae
  \begin{eqnarray*}
   &&{}_aD_x^{\alpha-1}(x-a)^{\nu}=\frac{\Gamma(\nu+1)(x-a)^{\nu-\alpha+1}}{\Gamma(\nu-\alpha+2)}, \quad \nu\in \mathcal{N}, \\
   &&(b-ax)^k_+=(b-ax)^k+(-1)^{k-1}(ax-b)_+^k,\quad k\in \mathcal{N}^+,
   \end{eqnarray*}
for $b/a\ge0, \, k\in \mathcal{N}^+$, it yields that
  \begin{eqnarray*}
  && {}_0 D_x^{\alpha-1}(ax-b)_+^k=a^{\alpha-1}\frac{\Gamma(k+1)}{\Gamma(k-\alpha+2)}(ax-b)_+^{k-\alpha+1},\\
   &&{}_0 D_x^{\alpha-1}(b-ax)_+^k=(-1)^{k-1}{}_0 D_x^{\alpha-1}(ax-b)_+^k+{}_0 D_x^{\alpha-1}(b-ax)^k.
   \end{eqnarray*}
 Hence, there exist
 \begin{eqnarray} \label{eq:formale-1}
 &&M^2(x):={}_0 D_x^{\alpha-1}\phi^{2}(x)=\frac{1}{\Gamma(3-\alpha)}\big(x_{+}^{2-\alpha}-2(x-1)^{2-\alpha}_{+}+(x-2)^{2-\alpha}_{+}\big),\\[5pt]\nonumber
 && M^{3}(x):={}_0 D_x^{\alpha-1}\phi^{3}(x)
 =\frac{1}{\Gamma(4-\alpha)}\left(x_{+}^{3-\alpha}-3(x-1)_{+}^{3-\alpha}+3(x-2)_{+}^{3-\alpha}-(x-3)_{+}^{3-\alpha}\right),\\[5pt]\nonumber
 &&M^{3,0}(x):={}_0 D_x^{\alpha-1}\phi^{3,0}(x)
        =\frac{1}{\Gamma(4-\alpha)}\left(-3x_{+}^{3-\alpha}+4(x-1)_{+}^{3-\alpha}-(x-2)^{3-\alpha}_{+}\right)+\frac{2}{\Gamma(3-\alpha)}x_+^{2-\alpha},\\[5pt]\nonumber
 &&M^{3,1}(x,l):={}_0 D_x^{\alpha-1}\phi^{3,0}(l-x)=\frac{2}{\Gamma(3-\alpha)}(x-l)_+^{2-\alpha}\\ \nonumber
 &&\qquad\qquad\qquad\qquad+\frac{1}{\Gamma(4-\alpha)}\left(3(x-l)^{3-\alpha}_{+}-4(x-l+1)^{3-\alpha}_{+}+(x-l+2)_{+}^{3-\alpha}\right),
\end{eqnarray}
and
\begin{eqnarray}\label{eq:formale-2}
&&\qquad\qquad\qquad\quad\qquad{}_0 D_x^{\alpha-1}\left(2^{J/2}\phi^{d}\left(2^Jx-k\right)\right)=2^{J(\alpha-1/2)}M^{d}\left(2^Jx-k\right),\quad d=2,\,3.\\[5pt]
&&{}_0 D_x^{\alpha-1}\left(2^{J/2}\phi^{3,0}(2^Jx)\right)=2^{J\left(\alpha-1/2\right)}M^{3,0}\left(2^Jx\right),\,\,
{}_0 D_x^{\alpha-1}\left(2^{J/2}\phi^{3,0}\left(2^J\left(1-x\right)\right)\right)=2^{J(\alpha-1/2)}M^{3,1}\left(2^Jx,2^J\right).\nonumber
\end{eqnarray}
Thus for both of the matrixes $A^d_l$ and $M$,  the main task is to generate their first columns which can be got by exactly calculating the inner products of the piecewise functions. The presented techniques actually apply to all the values of $d$.

For the general left and right Riemann-Liouville different matrixes, one has
 \begin{eqnarray}
&&A^d_l:=\left({}_0D_x^{\frac{\alpha}{2},0}\Phi_J^d,{}_xD_1^{\frac{\alpha}{2},0}\left(e^{-pUj\tau}\Phi_J^d\right)\right)=-\left({}_0D_x^{\alpha-1,0}\Phi_J^d,\frac{d}{dx}\left(e^{-pUj\tau}\Phi_J^d\right)\right), \label{eq:right11}\\
&&A^d_r:=\left({}_xD_1^{\frac{\alpha}{2},0}\Phi_J^d,{}_0D_x^{\frac{\alpha}{2},0}\left(e^{-pUj\tau}\Phi_J^d\right)\right)=\left({}_xD_1^{\alpha-1,0}\Phi_J^d,\frac{d}{dx}\left(e^{-pUj\tau}\Phi_J^d\right)\right),\quad d=2,\,3.\label{eq:right12}
\end{eqnarray}
Combining the symmetry of $\phi^d(x)$ with the technique given in Example 4.3 below, we obtain
\begin{eqnarray}\label{eq:formale-2}
&&{}_x D_1^{\alpha-1}\left(2^{J/2}\phi^{d}(2^Jx-k)\right)=2^{J\left(\alpha-1/2\right)}M^{d}\left(2^J\left(1-x\right)-\left(2^J-d-k\right)\right),\quad d=2,\,3,\\
&&{}_x D_1^{\alpha-1}\left(2^{J/2}\phi^{3,0}\left(2^Jx\right)\right)=2^{J\left(\alpha-1/2\right)}M^{3,1}\left(2^J\left(1-x\right),2^J\right),\\
&&{}_x D_1^{\alpha-1}\left(2^{J/2}\phi^{3,0}\left(2^J\left(1-x
\right)\right)\right)=2^{J\left(\alpha-1/2\right)}M^{3,0}\left(2^J\left(1-x\right)\right).
\end{eqnarray}
Therefore after an exact calculation of the fractional derivatives of the bases, the right hand terms in  (\ref{eq:right11}) and (\ref{eq:right12}) can be obtained by Gauss quadrature.

The nonlocal characteristics of fractional operator essentially lead to dense discrete system, which results in the heavy computation cost, especially for the high dimensional problem. By carefully choosing the bases, the potential quasi-Toeplitz structure can be dug out, which makes the cost of matrix vector product as $\mathcal{O}(N\log(N))$ ($N$ is the number of degrees of freedom) \cite{Wang:12}. At the same time, based on the wavelet bases it is easy to construct the preconditioner only with the cost $\mathcal{O}(N)$ \cite{Zhang:15}.
For example, for $d=2$, we define
\begin{eqnarray}
 &&\psi^{2,0}(x)=\frac{9}{10}\phi^2(2x-1)-\frac{3}{5}\phi^2(2x-2)+\frac{1}{10}\phi^2(2x-3),\\[4pt]\nonumber
&& \psi^{2}(x)=\phi^{2}(2x-1)-\frac{3}{5}\left(\phi^2(2x)+\phi^2(2x-2)\right)+\frac{1}{10}\left(\phi^2(2x+1)+\phi^2(2x-3)\right).\\ [4pt]\nonumber
&&\Psi_j^{2}:=\left\{2^{j/2}\psi_j^{2,0}(2^jx),\,2^{j/2}\psi^{d}(2^jx-k)\big|_{k=0}^{2^j-2},\,2^{j/2}\psi^{2,0}\left(2^j(1-x)\right)\right\}.
\end{eqnarray}
Then $W^2_j={\rm span}\{\Psi_j^{2}\}$ is the complement space of $S^2_j$ into $S^2_{j+1}$. And it is easy to obtain the corresponding refinement matrixes $M_{j,0}^2$ and $ M_{j,1}^2$ appeared in (\ref{Refine-raltion}), which can be implemented by FWT. Therefore the iteration method can be efficiently used to solve the corresponding matrix equation; a specific realizing  process for a two-dimension model will be given in Example 4.2.
\begin{example}
Consider (\ref{eq:model1}) on $(0,1)$ with $ K_{\gamma,\alpha}=-2\cos(\alpha\pi/2), U(x)=1$, $G_x(p,0)=sin(\pi x)$, and
\begin{eqnarray}
f(x,p,t)=\sigma e^{-pU(x) t} t^{\sigma-1}\sin(\pi x)-\left(\frac{\Gamma(\sigma+1)}{\Gamma(\sigma+\gamma)}t^{\sigma+\gamma-1}+\frac{t^{\gamma-1}}{\Gamma(\gamma)}\right)e^{-pU(x)t}P(x),
\end{eqnarray}
which is equivalent to (\ref{eq:nonsym}) with
\begin{eqnarray}
f(x,p,t)=\frac{\Gamma(\sigma +1)}{\Gamma(\sigma+1-\gamma)}e^{-pU(x) t}t^{\sigma-\gamma}\sin(\pi x)-e^{-pU(x)t}(t^{\sigma}+1)P(x),
\end{eqnarray}
and $P(x)={}_0D_x^{\alpha,0}\sin(\pi x)+{}_xD_1^{\alpha,0}\sin(\pi x)$. Then the exact solution is $G_x(p,t)=e^{-pU(x) t}\left(t^{\sigma}+1\right)\sin(\pi x)$.
\end{example}
The results listed in Tables \ref{tab:1-1}, \ref{tab:1-2}, and \ref{tab:1-2-2}, where (I, FBDF) denotes the FBDF scheme in Subsection 3.1 (Numerical schemes I), etc, confirm the theoretical convergence order w.r.t. time and space, respectively; especially the convergence rates in Table \ref{tab:1-2-2} also reflect the impact of the error term $\mathcal{O}\left(\tau^{\gamma}2^{J(\alpha-2d)}\right)$.  Note that the following techniques have been used to calculate the right vector term $\int_0^1f(x,t)\Phi_J(x)dx$:
\begin{eqnarray}\label{eq:Gauss}
P(x) &=&-{}_0D_x^{-(2-\alpha),0}\left(\pi^2\sin(\pi x)+ \sum_{k=1}^n\frac{(-1)^k \pi^{2k+1}x^{2k-1}}{(2k-1)!}\right)\\[5pt]\nonumber
&&-{}_xD_1^{-(2-\alpha),0}\left(\pi^2\sin(\pi(1- x))+\sum_{k=1}^n\frac{(-1)^k \pi^{2k+1}(1- x)^{2k-1}}{(2k-1)!}\right)\\[5pt]\nonumber
&&+\sum_{k=0}^n\frac{(-1)^k \pi ^{2k+1}}{\Gamma(2k+2-\alpha)}(1-x)^{2k+1-\alpha}+\sum_{k=0}^n\frac{(-1)^k \pi ^{2k+1}}{\Gamma(2k+2-\alpha)}x^{2k+1-\alpha},\, n\ge 1;
\end{eqnarray}
\begin{eqnarray}
&& \int_0^b\left({}_0D_x^{-(2-\alpha),0}h_1(x)\right)\,h_2(x)dx \\[3pt]\nonumber
&& =\frac{2^{\alpha-5}b^{3-\alpha}}{\Gamma(2-\alpha)}\int_{-1}^1\underbrace{(\xi+1)^{2-\alpha}}_{weight}h_2\left(\frac{\xi+1}{2}b\right)
\int_{-1}^1\underbrace{(1-\eta)^{1-\alpha}}_{weight}h_1\left(\frac{(\eta+1)(\xi+1)}{4}b\right)d\eta d\xi;
\end{eqnarray}

\begin{eqnarray}
&&\int_a^1\left({}_xD_1^{-(2-\alpha),0}h_1(x)\right)\,h_2(x)dx=
\int_{-1}^1\underbrace{(1-\xi)^{2-\alpha}}_{weight}h_2\left(\frac{1+a}{2}+\frac{1-a}{2}\xi\right)\\\nonumber
&&\qquad\quad  \times\int_{-1}^1\underbrace{(1+\eta)^{1-\alpha}}_{weight}h_1\left(\frac{1+\eta}{2}+\frac{1-\eta}{2}\left(\frac{1+a}{2}+\frac{1-a}{2}\xi\right)\right)d\eta d\xi\times \frac{2^{\alpha-5}(1-a)^{3-\alpha}}{\Gamma(2-\alpha)};\\[5pt]\nonumber
&& \int_a^b h(x)dx=\left(\int_0^b-\int_0^a\right)h(x)dx=\left(\int_a^1-\int_b^1\right)h(x)dx.
\end{eqnarray}
Thus the integrals of the first two and the last two terms in (\ref{eq:Gauss}) can be calculated by the Gauss-Lobatto-Jacobi quadrature and the analytical formula, respectively.
\begin{table}[!h t b p]\fontsize{6.0pt}{11pt}\selectfont
\begin{center}
 \caption{Convergence rate in time for Example 4.1 with $p=3, \sigma=2, \alpha=1.6,d=2,J=9$, and $T=0.5$.}
\begin{tabular}{ cc|c c |c c|c c|c c}
  \hline
  $\gamma$ &$N$    &\multicolumn{2}{c|}{($\mathrm{I}$, FBDF)}  &\multicolumn{2}{c|}{($\mathrm{I}$, PI)}   & \multicolumn{2}{c|}{($\mathrm{II}$, FBDF)} & \multicolumn{2}{c}{($\mathrm{II}$, PI)} \\
  \cline{3-8}\cline{9-10}
       &         & Err-2     & Rate                  &Err-2  & rate                 &Err-2  & Rate                        & Err-2 &Rate \\
   \hline
          &$40$    & 6.5141e-04      & ---              & 1.2072e-04     & ---       &  5.9163e-05    & ---          &  5.4194e-06  & --- \\
  $0.4$   &$60$    & 4.3413e-04     & 1.0008           & 6.8519e-05     & 1.3968     &  3.9451e-05    &0.9994        & 2.8337e-06  &1.5992\\
          &$80$    & 3.2553e-04     & 1.0007           & 4.5816e-05     & 1.3990     &  2.9582e-05    &1.0008       &  1.7823e-06  &1.6118\\
\hline
         &$40$    & 3.5561e-04      & ---               &  1.1469e-05   &  ---       & 1.5873e-04     & ---          &  6.8923e-05  & --- \\
  $0.8$  &$60$    & 2.3722e-04     & 0.9985            &  5.5719e-06   & 1.7804      & 1.0591e-04     & 0.99796       & 4.2399e-05  &1.1983\\
         &$80$    & 1.7796e-04     & 0.9991            &  3.3288e-06   & 1.7906      &  7.9456e-05     &0.99893      & 3.0023e-05 & 1.1998\\
\hline
\end{tabular} \label{tab:1-1}
\end{center}
\end{table}
\begin{table}[!h t b]\fontsize{6.0pt}{11pt}\selectfont
\begin{center}
 \caption{Convergence rate in space for Example 4.1  with $p=3,\sigma=2,\alpha=1.2, \gamma=0.6$, and $T=0.5$.}
\begin{tabular}{cc|c c |c c|c c|c c}
  \hline
  $(N,d)$& $J$    &\multicolumn{2}{c|}{($\mathrm{I}$, FBDF)}  &\multicolumn{2}{c|}{($\mathrm{I}$, PI)}   & \multicolumn{2}{c|}{($\mathrm{II}$, FBDF)} & \multicolumn{2}{c}{($\mathrm{II}$, PI)} \\
  \cline{3-8}\cline{9-10}
             &        & Err-1      & Rate                  &Err-1  & rate                 &Err-1  & Rate                        & Err-1 &Rate \\
   \hline
            &$4$  &  8.1510e-04    & ---               & 8.4143e-04   &  ---       & 7.4549e-04    & ---            & 7.6267e-04  & --- \\
$(2^{2J},2)$& $5$ &  2.0346e-04    & 2.0022            &2.0926e-04     &2.0075     &1.8564e-04     & 2.0057          & 1.8992e-04 &2.0057 \\
            &$6$  & 5.0855e-05     & 2.0003            & 5.2191e-05   &2.0035     & 4.6337e-05     & 2.0023          & 4.6337e-05  &2.0024\\
\hline
             &$3$  & 1.6566e-04   & ---                & 1.1864e-04    &  ---       & 1.2261e-04    & ---          &   1.1092e-04  & --- \\
$(2^{3J},3)$ &$4$  &1.9311e-05     & 3.1007            & 1.2894e-05   & 3.2018      & 1.3849e-05    & 3.1463        & 1.2171e-05  & 3.1879\\
             &$5$ &2.2987e-06    & 3.0705              & 1.4502e-06   &  3.1524     & 1.6078e-06     & 3.1066        &  1.3834e-06  & 3.1372\\
 \hline
\end{tabular} \label{tab:1-2}
\end{center}
\end{table}
\begin{table}[!h t b]\fontsize{6.0pt}{11pt}\selectfont
\begin{center}
 \caption{Convergence rate in space for Example 4.1  with $p=3,\sigma=2,\alpha=1.2$, and $T=0.5$.}
\begin{tabular}{cc|c c |c c|c c|c c}
  \hline
   & $J$    &\multicolumn{2}{c|}{$(d=2,\gamma=0.4,N=2^{2J})$}  &\multicolumn{2}{c|}{$(2,0.7,2^{2J})$}   & \multicolumn{2}{c|}{$(3,0.2,2^{3J})$} & \multicolumn{2}{c}{$(3,0.4,2^{3J})$} \\
  \cline{3-8}\cline{9-10}
             &        & Err-1      & Rate                  &Err-1  & rate                 &Err-1  & Rate                        & Err-1 &Rate \\
   \hline
                    &$3$  & 4.8153e-03    & ---              & 2.8744e-03     &  ---       & 3.3202e-04   & ---           &  1.9048e-04 & --- \\
($\mathrm{I}$, PI)  & $4$ &1.3645e-03     & 1.8193           & 6.7345e-04     &2.0936      & 5.0044e-05    & 2.7300       & 2.3496e-05  &3.0192\\
                    &$5$ &  3.8919e-04     & 1.8098          &1.5912e-04     & 2.0814      & 7.6693e-06     &2.7060        & 2.9378e-06  & 2.9996\\
\hline
                     &$3$  &4.3742e-03    & ---               &  2.5524e-03    &  ---       & 3.1818e-04    & ---          & 1.8400e-04  & --- \\
($\mathrm{II}$, FBDF)&$4$  &1.2314e-03    & 1.8287            &  5.9800e-04    & 2.0937    & 4.8008e-05      & 2.7285       & 2.2717e-05  &3.0179\\
                     &$5$  & 3.4975e-04    & 1.8159            & 1.4137e-04    & 2.0807     & 7.3572e-06     &2.7060        &  2.8386e-06  &3.0005\\
 \hline

\end{tabular} \label{tab:1-2-2}
\end{center}
\end{table}

In \cite{Deng:07}, the author has developed the non-uniform discretization to the time fractional derivative. Here we point out that this discretization can also be used in the schemes of this paper to improve the convergence or reduce the impact of initial singularity. Given any mesh $0=t_0<t_1<\cdots,<t_n<\cdots<t_N=T$, based on (\ref{eq:PI}) it yields a discrete approximation of  ${}^SD_{*,t}^{\gamma}v(t)$, provided as
\begin{equation}
\tilde{\mathcal{L}}^{\gamma}V_n=\sum_{l=0}^{n-1}e^{2pU( t_{n-l}-t_n )}\left(\tilde{Q}_l^{\gamma}-\tilde{Q}_{l-1}^{\gamma}\right)V_{n-l}-e^{-pUt_n}\tilde{Q}_{n-1}^{\gamma}V_0,
\end{equation}
where
\begin{equation}
\tilde{Q}^{\gamma}_{j}={\Gamma(1-\gamma)}^{-1}\int_{t_{n-1-j}}^{t_{n-j}}(t_n-s)^{-\gamma}\tau_{n-1-j}^{-1}ds,\quad j=0,\cdots,n-1, ~ \tilde{Q}_{-1}^\gamma=0.
\end{equation}
Then it holds that
\begin{eqnarray}
\tilde{Q}^{\gamma}_0>\tilde{Q}^{\gamma}_1>\cdots>\tilde{Q}^{\gamma}_{n-1}>0\quad\mbox{and}
\quad \tilde{L}_n^{\gamma}:=\frac{t_n^{-\gamma}}{\Gamma(1-\gamma)}\le\tilde{Q}^{\gamma}_{n-1}.
\end{eqnarray}
So one has the inequalities similar to (\ref{eq:sss3}) and (\ref{eq:lemma22_1}), respectively,  given as
\begin{eqnarray}
 &&V_n\cdot\tilde{\mathcal{L}}^{\gamma}V_n\ge\left(\frac{\tilde{Q}_0^{\gamma}}{2}+\frac{\tilde{Q}^{\gamma}_{n-1}}{4}\right)V_n^2+ \frac{1}{2}\sum_{l=1}^{n-1} e^{2pU( t_{n-l}-t_n )}\left(\tilde{Q}^{\gamma}_l-\tilde{Q}^{\gamma}_{l-1}\right)V_{n-l}^2-e^{-2pU t_n}\tilde{Q}^{\gamma}_{n-1}V_0^2,\\[5pt]
&&\qquad\qquad\qquad\qquad\sum_{n=0}^{M-1}\left(\tilde{\mathcal{L}}^{1-\gamma}\, V_n+e^{-pU t_n}\,\tilde{L}^{1-\gamma}_n\,V_0\right)V_n\ge 0,\quad M=1,2,\cdots,N,
\end{eqnarray}
where the diagonal dominance of the quadratic form matrix and the Gershgorin circle theorem ensure the last inequality holds.
The primal PI scheme $\mathrm{I}$ needs to be made some changes, denoted as $\mathrm{I}'$, i.e., replace (\ref{eq:L1-1}) and (\ref{eq:L1-21}) with $\mathcal{L}_1G_J^n=\frac{G_J^n-e^{pU(t_{n-1}-t_n)}G_J^{n-1}}{\tau}$ and $\mathcal{L}_2G_J^n=\tilde{\mathcal{L}}^{1-\gamma}G_J^n+\frac{e^{-pUt_n}G_J^0}{t_n^{1-\gamma}\Gamma(\gamma)}$, respectively. The primal PI scheme $\mathrm{II}$ keeps the original form but with variable time sizes. Both the full discrete schemes  (\ref{eq:fulldiscrete}) and (\ref{eq:sss5}) with non-uniform time stepping are still stable; and their proofs are similar to the ones of Theorems \ref{eq:theorem1} and \ref{theorem:22}. The numerical results with uniform mesh and grad mesh $t_j=(j/N)^{\upsilon}T,\,\upsilon>1, j=0,1,\cdots,N$ are listed in Table \ref{tab:1-4}, which show the advantages of non-uniform discretizaion.


%
\begin{table}[!h t b]\fontsize{6.0pt}{12pt}\selectfont
\begin{center}
 \caption{The $L_2$ numerical errors for Example 4.1 with  the uniform and grad time mesh, where $\alpha=1.7,\sigma=0.3,\nu=2, \gamma=0.8, d=2, J=9$, and $T=0.5$. }
\begin{tabular}{ cc|c  |c |c |c }
  \hline
     &$N$    &($\mathrm{I}'$, PI), $p=0$ &($\mathrm{I}'$, PI), $p=5$  & ($\mathrm{II}$, PI), $p=0$ & ($\mathrm{II}$, PI), $p=5$\\
\hline
             &$60$    & 4.9355e-04         &  4.0513e-05            &  7.0536e-03             &5.7899e-04    \\
    Uniform  &$80$    & 4.0772e-04          & 3.3468e-05            &  6.2912e-03          &5.1642e-04   \\
             & $100$  &3.5354e-04           &  2.9020e-05           & 5.7783e-03         & 4.7431e-04  \\
 \hline
            &$60$    & 2.5797e-04           &  2.1176e-05            & 3.6148e-03            & 2.9672e-04     \\
    Grad    &$80$    & 1.8862e-04           &   1.5483e-05          &  3.0007e-03        & 2.4632e-04  \\
            &$100$   & 1.4824e-04           &  1.2168e-05              &2.6035e-03          & 2.1371e-04  \\
 \hline

\end{tabular} \label{tab:1-4}
\end{center}
\end{table}

\begin{example} We consider the two-dimensional problem
\begin{eqnarray}\label{eq:example2}
{}^SD_{*,t}^\gamma G_{x,y}(p,t)={}_0D_x^{\alpha,0}G_{x,y}(p,t)+{}_0D_y^{\alpha,0}G_{x,y}(p,t)+f(x,y,p,t),\nonumber
\end{eqnarray}
on $(0,1)\times(0,1)$ with the initial condition $G_{x,y}(p,0)=xy^2(1-x)^2(1-y)$ and homogeneous boundary condition, and
\begin{eqnarray}
f(x,y,p,t)&=&\frac{e^{-pU t} t^2}{\Gamma(3-\alpha)}\bigg[\Big((2-\alpha)x^{1-\alpha}-4x^{2-\alpha}+\frac{6}{3-\alpha}x^{3-\alpha}\Big)y^2(1-y) \\[3pt]
&&+\Big(2y^{2-\alpha}-\frac{6}{3-\alpha}y^{3-\alpha}\Big)x(1-x)^2\bigg]+\frac{\Gamma(3+\gamma)}{2}t^2e^{-pU  t}x(1-x)^2y^2(1-y).\nonumber
\end{eqnarray}
It is easy to see that the exact solution is $G(x,y,t)=e^{-pU t}t^{2+\gamma} x(1-x)^2y^2(1-y)$.
\end{example}

In space approximation, we  define $S^{2,d}_j:={\rm span}\left\{\Phi^{2,d}_j(x,y)\right\},\,\Phi^{2,d}_j(x,y):=\Phi^d_j(x)\otimes\Phi^d_j(y)$. Then $\{S^{2,d}_j\}$ constitutes a MRA of $L_2(\Omega)$ with  the refinement relation:
\[\Phi^{2,d}_j(x,y)^T=\Phi^{2,d}_{j+1}(x,y)^TM^{2,d}_{j,0},\quad M^{2,d}_{j,0}=M^d_{j,0}\otimes M^d_{j,0}.\]
Letting $\Psi_j^{2,d}(x,y):=\left\{\Psi^d_j(x)\otimes \Phi^d_j(y),\Phi^d_j(x)\otimes\Psi^d_j(y),\Psi^d_j(x)\otimes\Psi^d_j(y)\right\}$,
 one can define the corresponding wavelet spaces by $W^{2,d}_j:=S^d_{j+1}\cap (S^{d}_j)^T = {\rm span} \left\{\Psi_j^{2,d}(x,y)\right\}$, and
it holds that
\begin{equation}
\Psi_j^{2,d}(x,y)^T=\Phi^{2,d}_{j+1}(x,y)^TM^{2,d}_{j,1}, \quad\rm{where}\,\,
 M^{2,d}_{j,1}=\left(M^d_{j,1}\otimes M^d_{j,0},\,M^d_{j,0}\otimes M^d_{j,1},\,M^d_{j,1}\otimes M^d_{j,1}\right).
 \end{equation}
Then $\left\{\Phi^{2,d}_{J_0}(x,y),\Psi^{2,d}_{J_0}(x,y),\cdots,\Psi^{2,d}_{J-1}(x,y)\right\}$ forms a multiscale base of $S^{2,d}_J$; and the relations (\ref{Refine-raltion})  still hold with $M_j:=\left(M^{2,d}_{j,0},M^{2,d}_{j,1}\right)$,  which   satisfy the same properties as their one-dimension version. The full discrete MGM of (\ref{eq:example2}) being similar to (\ref{eq:sss5}) reads as:
 find $G_J^n=P_J^{n,d}\Phi_J^{2,d}\in S_J^{2,d}\subset H_0^{\frac{\alpha}{2}}(\Omega)$ such that
\begin{eqnarray} \label{eq:twodimension}
\qquad (\mathcal{L}G^n_J,\,v)+\left({}_0D_x^{\frac{\alpha}{2},0}G_J^n,{}_xD_1^{\frac{\alpha}{2},0}v\right)+\left({}_0D_y^{\frac{\alpha}{2},0}G_J^n,{}_yD_1^{\frac{\alpha}{2},0}v\right)=(f,\,v) \quad\forall v\in S^{2,d}_J,
\end{eqnarray}
where $\mathcal{L}G^n_J$ is given in (\ref{eq:discrete}), including the PI and FBDF cases.
It results in the similar theoretical results as given in Theorem \ref{theorem:22}.
Denoting
\begin{eqnarray}
 B^{2,d}_J&=&Q_0^{\gamma}M^d_J\otimes M^d_J+M_J^d\otimes A_l^d+ A_l^d\otimes M_J^d,\\\nonumber
F^{2,d}_J&=&\left(f,\Phi^{2,d}_J\right)+\sum_{l=1}^{n-1}e^{-pU l \tau}\left(Q_{l-1}^{\gamma}-Q_l^{\gamma}\right)\left(M^d_J\otimes M^d_J\right)P_J^{n-l,d}+Q_{n-1}^{\gamma}\left(M^d_J\otimes M^d_J\right)P_J^{0,d},\nonumber
\end{eqnarray}
where $M_J^d$ denotes the mass matrix, (\ref{eq:twodimension}) can be rewritten as
\begin{eqnarray}\label{eq:twomen}
 B^{2,d}_JP^{n,d}_J=F^{2,d}_J.
\end{eqnarray}
Note that $(M_J^{d}\otimes M_J^d) p^{l,d}_J$ can be calculated with the cost $\mathcal{O}(2^{2J})$, and
\begin{eqnarray}\label{eq:Matrixform}
\qquad\left(M_J^d\otimes A_l^d\right)P^{l,d}_J=A_l\mathcal{R}(P_J^{l,d})(M^d_J)^T,\quad  \left(A_l^d\otimes M_J^d \right)P^{l,d}_J=M_J^d\mathcal{R}(P_J^{l,d})(A_l^d)^T,
\end{eqnarray}
where $\mathcal{R}(P^{l,d})$ denotes the reshaped matrix of vector $P_J^{l,d}$. Thanks to the quasi-Toeplitz structure of $A_l^d$ and the sparse structure of $M_J^d$, both of the matrix vector products in (\ref{eq:Matrixform}) can be calculated with the cost $\mathcal{O}(J2^{2J})$. Thus the iteration method can be effectively used to solve (\ref{eq:twomen}).
\begin{table}[!h t b]\fontsize{6.0pt}{12pt}\selectfont
\begin{center}
 \caption{ Numerical results of the PI scheme of (\ref{eq:twodimension}) for Example 4.2, solved by the GMRES (restarted) and Bi-CGSTAB, with $pU=0, N=2^{2J}, d=3, J_0=2$, and $T=0.5$.}
\begin{tabular} {cc|c c c|c c c|c c}  \hline
  $(\alpha,\gamma)$        & $J$    &\multicolumn{3}{c|}{GMRES (30)}  &\multicolumn{3}{c|}{Bi-CGSTAB} &\multicolumn{2}{c}{Gauss} \\\cline{3-10}
             &         & Iter&  CPU(s)  &Err-2      & Iter& CPU(s) &Err-2                        & CPU(s) & Err-2         \\
      \hline
              & $4$  &46&  4.9419&  7.5237e-06   &30.3   &   3.0249 &7.5237e-06         & 0.7319  & 7.5237e-06    \\
$(1.5,0.5)$& $5$  &45&   31.0692&  5.2053e-07   &28.4   &  22.3599 &5.2053e-07         & 41.7140  & 5.2053e-07      \\
                 & $6$    &49&   691.7970& 3.4986e-08   &36.3    &  528.3210 & 3.4983e-08        & 9.3463e+03  & 3.4982e-08  \\
      \hline
    &$4$       & 21    & 2.0627 & 4.0444e-06   &12.5&   1.7651&   4.0444e-06         & 0.8240  & 4.0444e-06   \\
$(1.8,0.6)$& $5$      &  28   &  20.0747 & 2.5297e-07  &18.0&  17.7395&   2.5297e-07         & 43.1349 & 2.5297e-07       \\
           & $6$      &  41   & 486.7180  & 1.5677e-08  &28.0&   493.3244 &  1.5677e-08         &  4.5034e+03  & 1.5677e-08    \\
    \hline
   \end{tabular}\label{tab:2-1}
\end{center}
\end{table}
 The norm equivalence similar to (\ref{normequiv}) ensures that the following system is well conditioned:
\begin{eqnarray}
 \left(DM^TB^{2,d}_JMD\right)D^{-1}M^{-1}P^{n,d}_J=DM^T F^{2,d}_J,
\end{eqnarray}
where $MP^{l,d}, M^TP_J^{l,d}$ can be calculated by  $2$-dimension FWT with the cost $\mathcal{O}(2^{2J})$ \cite{Cohen:00,Urban:09}, and $D$ donotes the diagonal matrix obtained by the inverse  square root of the diagonal of $M^TB^{2,d}_JM$,  using the translation property of the same level wavelet bases and tensor-like structure of the multiscale base of $S_J^{2,d}$, which can be calculated with the cost $\mathcal{O}(J)$. For $d=3$ and $d=2$, the numerical results are listed in Table \ref{tab:2-1} and Table \ref{tab:2-2}, respectively, where `Iter' denotes the average iterations, and `Gauss' denotes the Gaussian elimination method. For fixed $t_n$, the initial iteration vector is chosen as the approximation solution at $t_{n-1}$, and the stopping criterion  is
 \[ \frac{\|r(k)\|_{l_2}}{\|r(0)\|_{l_2}}\le 1e-8,\]
 with $r(k)$ being the residual vector of linear systems after $k$ iterations.
 \begin{table}[!h t b p]\fontsize{6.0pt}{12pt}\selectfont
\begin{center}
 \caption{Numerical results of the FBDF scheme of (\ref{eq:twodimension}) for Example 4.2, solved by the GMRES (restarted) and Bi-CGSTAB (before/after preconditioned), with $pU=2, \alpha=1.5,  N=2^{J}, d=2, J_0=1$, and $T=0.5$.}
\begin{tabular}{C{0.8cm}C{0.1cm}|C{0.5cm}C{1cm}|C{0.5cm}C{0.6cm}|C{0.5cm}C{1cm}|C{0.5cm}C{0.6cm}|C{1.2cm}C{1.2cm}}  \hline
  $\gamma$        & $J$    &\multicolumn{2}{c|}{GMRES (30)}  &\multicolumn{2}{c|}{Pre-GMRES (30)} &\multicolumn{2}{c|}{Bi-CGSTAB}  &\multicolumn{2}{c|}{Pre-Bi-CGSTAB}&\multicolumn{2}{c}{Gauss} \\ \cline{3-12}
             &         & Iter&  CPU(s)      & Iter& CPU(s)      & Iter  & CPU(s)     & Iter  & CPU(s)   &  CPU(s)& Err-2 \\
      \hline
             & $5$    &  88    & 3.0068    &35&  1.3752   &44.4   & 2.0079      &21     & 1.3134  &1.6351 &5.8984e-06 \\
 $0.3$ & $6$   & 166.3  & 27.5488   &38&  9.6810   &95.7   & 20.1165     &23     & 8.1409  &58.3051&2.8107e-06    \\
            & $7$    &  334.9  &  372.4754 &40& 70.9539   &203.8  & 311.6926    &24    & 68.1188  & 4.1827e+03&1.4183e-06\\
      \hline
            &$5$    & 58.1 & 1.6590   &  32   & 1.2645    &35.0    & 1.5070   &18.5&   1.1204            &1.3775&2.4580e-05\\
$0.8$& $6$   & 94.2 & 18.5311  &  35   & 8.9110     &65.5     & 12.9719  &20.9&   7.5250           &57.4800&1.2492e-05    \\
           & $7$   & 155.7 &176.7667  &  35   & 63.1512    &136.4     & 180.8864  &20.9&   59.2873         &3.7824e+03& 6.3009e-06  \\
    \hline
   \end{tabular}\label{tab:2-2}
\end{center}
\end{table}
\begin{example}
We further consider the equation
\begin{eqnarray}\label{eq:1.1-1}
\frac{\partial}{\partial t}G_{x}(p,t)=K_{\gamma,\alpha}\, {}^SD_t^{1-\gamma}\,\nabla_{x}^{\alpha} G_{x}(p,t)-pU(x)G_{x}(p,t)+f(x,p,t),\quad 0<t\le T,
\end{eqnarray}
on $(0,1)$ with $T=0.5,\, K_{\gamma,\alpha}=-2\cos(\frac{\alpha\pi}{2})$, and $ U(x)=x$. The forcing term and initial condition can be derived from the exact solution $G_{x}(p,t)=\left(t^{\sigma}+5\right)e^{-pxt}\left(x^3-x\right)$.
\end{example}
The equivalent form of (\ref{eq:1.1-1}) is
\begin{eqnarray}
{}^SD_{*,t}^{\gamma}G_{x}(p,t)=K_{\gamma,\alpha}\nabla_{x}^{\alpha}G_{x}(p,t)+{}_0D_t^{-(1-\gamma),pU}f(x,p,t),\quad 0<t\le T.
\end{eqnarray}
Let $L^2(\Omega)$ and $H_0^{\frac{\alpha}{2}}(\Omega)$ be the complex value Sobolev spaces equipped with the corresponding complex inner product and norms. The fully discrete scheme can be given as: find $G_J^n\in\left(S^d_J+iS^d_J\right)\subset H_0^{\frac{\alpha}{2}}(\Omega)$ such that
\begin{equation}\label{eq:ee2}
\left({\mathcal{L}}G^n_J,\,v\right)+K_{\gamma,\alpha}B\left(G_J^n,\,v\right)=\left(f,\,v\right)\qquad \forall v\in S^d_J,
\end{equation}
where $\mathcal{L}G^n_J$ is given in (\ref{eq:discrete}), including the PI and FBDF cases. Using  $\mathcal{B}_J$ given in (\ref{eq:BB}) and defining $\mathcal{M}$ as the orthogonal projection operator: $\mathcal{M}u\in S_J^d$ and $\left(\mathcal{M} u -u,v\right)=0\, \,\forall v\in S_J^d$,
 then one can rewrite (\ref{eq:ee2}) as an operator form
\begin{equation}
G_J^n=\left({Q}^r_0+\mathcal{B}_J\right)^{-1}\left(\sum_{l=1}^{n-1}\left({Q}^{\gamma}_{l-1}-{Q}^{\gamma}_l\right)\mathcal{M}\left(e^{ -pUl}G_J^{n-l}\right)+{Q}^{\gamma}_{n-1}\mathcal{M}\left(e^{-pUn\tau}G_J^{0}\right)+\mathcal{M} f\right).
\end{equation}
 Thus, if $\Re\,(pU)\ge 0$ it holds that
\begin{equation}
\left\|G_J^n\right\|\le\left\|\left({Q}^r_0+\mathcal{B}_J\right)^{-1}\right\|\left(\sum_{l=1}^{n-1}\left({Q}^{\gamma}_{l-1}-{Q}^{\gamma}_l\right)
\left\|G_J^{n-l}\right\|+{Q}^{\gamma}_{n-1}\left\|G_J^{0}\right\|+\left\|\mathcal{M} f\right\|\right).
\end{equation}
Since $B_J$ is a selfadjoint positive define elliptic operator,
\begin{equation}
\left\|\big({Q}^r_0+\mathcal{B}_J\big)^{-1}\right\|=\sup_{\lambda >0}\left|\left({Q}^r_0+\lambda\right)^{-1}\right|< {1}/{{Q}^r_0}.
\end{equation}
Using the mathematical induction similar to the proof of Theorem \ref{theorem:22}, it is easy to show that the presented scheme (\ref{eq:ee2}) is unconditionally stable. The convergence rates for time and space  are given in Table \ref{tab:3-1} and  Figure \ref{Fig1.sub-1}, respectively, which accord with the theoretical results given in Theorem \ref{theorem:22}.
\begin{table}[!h t b]\fontsize{6.0pt}{12pt}\selectfont
\begin{center}
 \caption{Convergence rate in time of (\ref{eq:ee2}) for Example 4.3 with $K_{\gamma,\alpha}=-2\cos(\alpha\pi/2),p=1+i, i=\sqrt{-1},\, U(x)=x,\sigma=2,d=3$, and $T=0.5$.}
\begin{tabular} {cc|cc|cc|cc|c|c}  \hline
  $(\alpha,J)$        & $N$    &\multicolumn{2}{c|}{$\gamma=0.5$, FBFD }  &\multicolumn{2}{c|}{$\gamma=0.5$, PI} &\multicolumn{2}{c|}{$\gamma=0.8$, FBFD}  &\multicolumn{2}{c}{$\gamma=0.8$, PI}\\\cline{3-10}
             &         & Err-2&  Rate      & Err-2& Rate                               & Err-2&  Rate      & Err-2& Rate   \\
      \hline
             & $20$    & 9.9484e-05  &---    &  1.8120e-05  &---             & 1.9401e-04    &---        &9.4674e-05     &---    \\
 $(2,8)$    & $40$     & 4.9881e-05 &0.9960     & 6.4990e-06  & 1.4793       & 9.7205e-05  &0.9970     & 4.1306e-05    & 1.1966    \\
            & $60$     & 3.3285e-05 &0.9978  &  3.5599e-06  &1.4845          & 6.4847e-05  &0.9983      & 2.5415e-05    &1.1979 \\
      \hline
            &$20$      & 2.9375e-04   &---      &5.4765e-05  &---              &  5.8284e-04    &---       &2.9075e-04  &---     \\
$(1.6,6)$& $40$     &1.4734e-04   &0.9955   &1.9697e-05   &1.4753            & 2.9237e-04    &0.9953       &1.2705e-04  &1.1944      \\
           & $60$      & 9.8339e-05  &0.9971   &1.0819e-05  & 1.4777             & 1.9514e-04    &0.9972    &7.8226e-05   &1.1962    \\
    \hline
   \end{tabular}\label{tab:3-1}
\end{center}
\end{table}
\begin{figure}[!h t b p]
\centering
\includegraphics[width=0.45\textwidth]{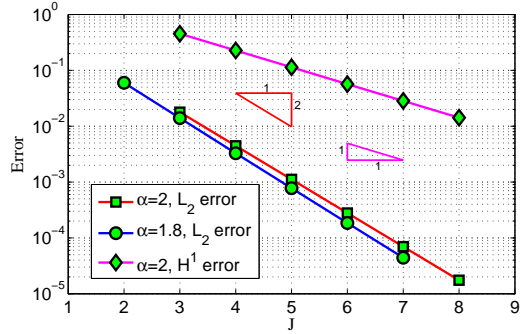}
\caption{Convergence rate in space of the FBDF scheme of (\ref{eq:ee2}) for Example 4.3, where $\sigma=2, p=1+i, i=\sqrt{-1},\,\gamma=0.5, N=2^{2J}$, and $d=2$.}    \label{Fig1.sub-1}
\end{figure}

Here, we further solve (\ref{eq:1.1-1}) by B-spline collocation method, which is convenient to treat the case with variable coefficients. For $d=4$, denote  
  \begin{eqnarray}
  \phi^4(x)&=&\frac{1}{6}\sum_{i=0}^4{4 \choose i}(-1)^i(x-i)_+^3,\\ \label{Cubic:1}
   \phi^{4,0}(x)&=&3x_+-\frac{9}{2}x_+^2+\frac{7}{4}x_+^3-2(x-1)_+^3+\frac{1}{4}(x-2)_+^3,\\
   \phi^{4,1}(x)&=&\frac{3}{2}x_+^2-\frac{11}{12}x_+^3+\frac{3}{2}(x-1)_+^3-\frac{3}{4}(x-2)_+^3+\frac{(x-3)^3_+}{6}. \label{Cubic:2}
	\end{eqnarray}
Then $\Phi_J=2^{J/2}\left\{\phi^{4,0}(2^Jx),\,\phi^{4,1}(2^Jx), \,\phi^4(2^Jx-k) \left| _{k=0}^{2^J-4}\right.,\, \phi^{4,1}\left(2^J(1-x)\right),\, \phi^{4,0}\left(2^J(1-x)\right)\right\}$ is a Riesz base of $S_J$, and the collocation method reads: find $G_J^n\in S_J$ such that
\begin{equation} \label{eq:4.41}
\big(\mathcal{L}_1G_J^n-K_{\gamma,\alpha}\mathcal{L}_2\left(\nabla_x^{\alpha}G_J^n\right)-f\big)\Big|_{x_i}=0 \quad \forall x_i\in \left\{{1}/{2^{J+1}},\,{k}/{2^J}\left|_{k=1}^{2^J-1}\right.,\,1-{1}/{2^{J+1}}\right\},
\end{equation}
where $\mathcal{L}_1G_J^n$ and $\mathcal{L}_2G_J^n$ are given in (\ref{eq:L1-1}) and (\ref{eq:L1-21}), i.e., the PI case.
 In the implementation, one needs to calculate the left and right fractional differential matrix $A_{L}:={}_0 D_x^{\alpha}\Phi_J\big|_{x_i}$ and $A_{R}:={}_x D_1^{\alpha}\Phi_J\big|_{x_i}$. The former can be got by first getting the formula of ${}_0 D_x^{\alpha}\Phi_J$ similar to (\ref{eq:formale-1}) and (\ref{eq:formale-2}),  then use the fact that the  matrix produced by $2^{J/2}{}_0 D_x^{\alpha}\phi(2^Jx-k)\big|_{i/2^J},\,k,i=2,3,\cdots,2^J-2$ has a Toeplitz structure; and the latter can be easily obtained by using the relationship $A_{R}=A_{L}({\rm end}:-1:1,{\rm end}:-1:1)$. In fact, for $v(x)$ satisfying $v(0)=v(1)=0$, there exists
\begin{eqnarray}\label{eq:compute-1}
{}_0D_x^{\alpha}v(x)&=&{}_0D_x^{-(2-\alpha)}v^{\prime\prime}(x)+\frac{v^{\prime}(0)}{\Gamma(2-\alpha)}x^{1-\alpha},\\\nonumber
{}_xD_1^{\alpha}v(x)&=&{}_xD_1^{-(2-\alpha)}v^{\prime\prime}(x)-\frac{v^{\prime}(1)}{\Gamma(2-\alpha)}(1-x)^{1-\alpha}.
\end{eqnarray}
For $k=0,1,\cdots,2^J-4$,
 \begin{eqnarray}
 {}_0 D_x^{\alpha}\phi^4(2^Jx-k)&=&\frac{2^{2J}}{\Gamma(2-\alpha)}\int_0^x(x-\xi)^{1-\alpha}\left(\phi^4\right)^{\prime\prime}(2^J\xi-k)\mathrm{d}\,\xi\\\nonumber
                               &=&\frac{2^{J\alpha}}{\Gamma(2-\alpha)}\int_0^{2^Jx-k}(2^Jx-\xi-k)^{1-\alpha}\left(\phi^4\right)^{\prime\prime}(\xi)\mathrm{d}\,\xi,
 \end{eqnarray}
 \begin{eqnarray}
 {}_x D_1^{\alpha}\phi^4\left(2^Jx-(2^J-4-k)\right)&=&\frac{2^{2J}}{\Gamma(2-\alpha)}\int_x^1(\xi-x)^{1-\alpha}\left(\phi^4\right)^{\prime\prime}(2^J-2^J\xi-k)\mathrm{d}\,\xi\\\nonumber
                               &=&\frac{2^{J\alpha}}{\Gamma(2-\alpha)}\int_0^{2^J-2^Jx-k}(2^J-k-\xi-2^Jx)^{1-\alpha}\left(\phi^4\right)^{\prime\prime}(\xi)\mathrm{d}\,\xi,
 \end{eqnarray}
then
\begin{equation}
{}_0 D_y^{\alpha}\phi^4(2^Jy-k)\Big|_{y=x}={}_y D_1^{\alpha}\phi^4(2^Jy-(2^J-4-k))\Big|_{y=1-x}, \quad x,y\in(0,1),
\end{equation}
where the properties $\phi^4(x)=\phi^4(4-x)$ and ${\rm supp}\phi^4(x)=(0,4)$ are used. 
Noticing that ${\rm supp}\phi^{4,1}(x)=(0,3),\, \phi^{4,1}(0)=\left(\phi^{4,1}\right)^{\prime}(0)=0$ and ${\rm supp}\phi^{4,0}(x)=(0,2),\,\phi^{4,0}(0)=0$, it is also easy to obtain
\begin{eqnarray} \label{eq:compute-2}
{}_0 D_y^{\alpha}\phi^{4,1}(2^Jy)\Big|_{y=x}&=&{}_y D_1^{\alpha}\phi^{4,1}\left(2^J(1-y)\right)\Big|_{y=1-x}, \quad x,y\in\left(0,1\right),\\\nonumber
{}_0 D_y^{\alpha}\phi^{4,0}(2^Jy)\Big|_{y=x}&=&{}_y D_1^{\alpha}\phi^{4,0}\left(2^J(1-y)\right)\Big|_{y=1-x}, \quad x,y\in\left(0,1\right).
\end{eqnarray}
When performing the numerical computation, the exact solution of Example 4.3 is replaced by $G_{x}(p,t)=(t^{\sigma}+5)(x^3-x^4)$.  The numerical results are presented in Table \ref{tab:3-3}, which show a ($4-\alpha$)-th order convergence in space.
\begin{table}[htbp]\fontsize{7.0pt}{11pt}\selectfont
\begin{center}
 \caption{Convergence rate (maximum norm) in space of the PI scheme of (\ref{eq:4.41}) for Example 4.3, where $K_{\gamma,\alpha}=-2\cos(\alpha\pi/2), U(x)=x,\sigma=2,N=2^{15},d=4$, and $T=0.5$.}
\begin{tabular} {c|cc|cc|cc|c|c}  \hline
   $J$    &\multicolumn{2}{c|}{$p=1+i,\alpha=2$}  &\multicolumn{2}{c|}{$p=10i,\alpha=2$} &\multicolumn{2}{c|}{$p=0,\alpha=1.3$}  &\multicolumn{2}{c}{$p=1+i,\alpha=1.6$}\\\cline{2-9}
        & maxErr&  Rate      & maxErr& Rate                               & maxErr&  Rate      & maxErr& Rate   \\
      \hline
         $3$    &  1.8039e-02  &---    &  1.7092e-02   &---             & 2.0181e-03 &---        & 5.1635e-03   &---    \\
         $4$    &   4.7323e-03  &1.9305  & 4.4882e-03   & 1.9291          & 3.2128e-04  & 2.6511    & 1.0086e-03   & 2.3560    \\
         $5$     &   1.1970e-03 &1.9831  & 1.1354e-03  &  1.9829          & 4.9939e-05  & 2.6856    & 1.8940e-04   & 2.4128 \\
    \hline
   \end{tabular}\label{tab:3-3}
\end{center}
\end{table}
\begin{example}
We solve the forward model
\begin{equation}\label{eq:forward}
\left\{\begin{array}{ll}
\frac{\partial}{\partial t}G(x,p,t)=K_{\gamma,\alpha}\,\nabla_x^{\alpha}\,{}^SD_t^{1-\gamma}G(x,p,t)-pU(x)G(x,p,t)+f(x,p,t), &0<t\le T,\\[5pt]
G(x,p,0)=g(x,p), &x\in (0,1),
\end{array}
\right.
\end{equation}
where the values of the parameters are the same as Example 4.3. 
\end{example}
The MGM of the model reads: find $G_J^n\in S_J$ such that
\begin{equation} \label{eq:4.48}
\left(\mathcal{L}_1G^n_J,v\right)+K_{\gamma,\alpha}B\left(\mathcal{L}_2G^n_J,v\right)=\left(f,v\right)\quad \forall v\in S_J,
\end{equation}
The two cases for (\ref{eq:4.48}) are simulated, i.e., the FBDF case with $\mathcal{L}_1G_J^n$ and $\mathcal{L}_2G_J^n$ defined in (\ref{eq:L11}), and the PI case with $\mathcal{L}_1G_J^n$ and $\mathcal{L}_2G_J^n$ defined in (\ref{eq:L1-1}) and (\ref{eq:L1-21}).
The mass matrix $\left(e^{-pUj\tau}\Phi_J,\phi_J\right)$ can be directly calculated by Gauss quadrature. To handle the terms $B\left(e^{-pUj\tau}\Phi_J,\,\Phi_J\right)$, it can be noticed that
\begin{eqnarray}
\left({}_0D_x^{\frac{\alpha}{2},0}\left(e^{-pU j\tau}\Phi_J\right),{}_xD_1^{\frac{\alpha}{2},0}\Phi_J\right)&=&\left(\frac{d}{dx}\left(e^{-pU j\tau}\Phi_J\right),{}_xD_1^{\alpha-1,0}\Phi_J\right),\label{eq:right13}\\ 
\left({}_xD_1^{\frac{\alpha}{2},0}\left(e^{-pU j\tau}\Phi_J\right),{}_0D_x^{\frac{\alpha}{2},0}\Phi_J\right)&=&-\left(\frac{d}{dx}\left(e^{-pU j\tau}\Phi_J\right),{}_0D_x^{\alpha-1,0}\Phi_J\right).\label{eq:right14}
\end{eqnarray}
Letting $\Phi_J^r={\rm fliplr(\Phi_J)}$ be the reverse rearrangement of $\Phi_J$, from (\ref{eq:formale-2}), one has
\begin{equation}
{}_xD_1^{\alpha-1,0}\Phi_J(x)={}_0D_x^{\alpha-1,0}\Phi_J^r(1-x).
\end{equation}
After an exact calculation of ${}_0D_x^{\alpha-1,0}\Phi^r_J$ and ${}_xD_1^{\alpha-1,0}\Phi^r_J$, one can obtain the matrixes corresponding to (\ref{eq:right13}) and (\ref{eq:right14}) by  Gauss quadrature. The convergence rates in time and space are listed in Tables \ref{tab:3-4} and \ref{tab:3-5}, respectively, which are good consistent with the theoretical results presented in Theorem \ref{theorem:11}.
\begin{table}[!h t b]\fontsize{7.0pt}{12pt}\selectfont
\begin{center}
 \caption{Convergence rate in time of (\ref{eq:4.48}) for Example 4.4 with $\sigma=2, \alpha=2,K_{\gamma,\alpha}=-2\cos(\alpha\pi/2),p=1+i, i=\sqrt{-1}, U(x)=x,d=3,J=9,$ and $T=0.5$.}
\begin{tabular} {c|cc|cc|cc|c|c}  \hline
   $N$    &\multicolumn{2}{c}{$\gamma=0.2$, FBFD }  &\multicolumn{2}{c|}{$\gamma=0.2$, PI} &\multicolumn{2}{c|}{$\gamma=0.6$, FBFD}  &\multicolumn{2}{c}{$\gamma=0.6$, PI}\\\cline{2-9}
        & Err-2&  Rate      & Err-2& Rate                               & Err-2&  Rate      & Err-2& Rate   \\
      \hline
         $120$    &  3.5841e-04   &---    &1.3009e-04  &---             &1.9803e-04  &---       & 9.9421e-06   &---   \\
         $150$    &  2.8670e-04   &1.0005  &9.9597e-05  &1.1971        &1.5842e-04  &1.0000    &6.9837e-06   &1.5829  \\
         $180$    &  2.3890e-04   &1.0004   &8.0061e-05  &1.1975              &1.3202e-04  &1.0000    & 5.2316e-06     &1.5844  \\
    \hline
   \end{tabular}\label{tab:3-4}
\end{center}
\end{table}
\begin{table}[!h t b]\fontsize{7.0pt}{12pt}\selectfont
\begin{center}
 \caption{Convergence rate in space of the FBDF scheme of (\ref{eq:4.48}) for Example 4.4 with $K_{\gamma,\alpha}=-2\cos(\alpha\pi/2),U(x)=x,\sigma=2,\gamma=0.5,N=2^{2J},T=0.5,$ and $d=2$.}
\begin{tabular} {c|cc|cc|cc|c|c}  \hline
   $J$    &\multicolumn{2}{c|}{$\alpha=1.2,p=1+i$}  &\multicolumn{2}{c|}{$\alpha=1.6,p=10+5i$} &\multicolumn{2}{c|}{$\alpha=1.8,p=5$}  &\multicolumn{2}{c}{$\alpha=2,p=i$}\\ \cline{2-9}
        & Err-2&  Rate      & Err-2& Rate                               & Err-2&  Rate      & Err-2& Rate   \\
      \hline
         $3$    & 1.1063e-02  &---    &   1.8155e-02   &---             & 1.2438e-02  &---       & 2.6211e-02     &---    \\
         $4$    & 2.5275e-03  &2.1300  &  4.3328e-03   &2.0670          & 2.9620e-03  &2.0701   & 6.5621e-03  & 1.9979     \\
         $5$    & 5.9327e-04  &2.0910  &  1.0042e-03  & 2.1093          & 6.9960e-04  &2.0820    & 1.6411e-03  & 1.9995 \\
    \hline
   \end{tabular}\label{tab:3-5}
\end{center}
\end{table}
\section{Conclusion}
A great number of small events usually lead to the Brownian motion, however the rare but large fluctuations may result in non-Brownian behavior. With the deep studying the non-Brownian motion, the functional distribution of the paths of particles performing anomalous diffusion attracts the interests of the researchers; and it has wide applications. This paper discusses  the model, characterizing the distribution of the functional of the paths of anomalous motion (described by the time-space fractional diffusion equation). The main efforts of this paper are to provide the efficient computation methods for the model. Two types of time discretizations are introduced for the fractional substantial derivative. And the multiresolution Galerkin methods are used for the space approximation with B-spline functions as the bases, which has the striking benefits of keeping the Toeplitz structure of the stiffness matrix. The unconditional stability and convergence are theoretically proved and numerically verified. In fact, in the section of numerical experiments, more detailed numerical techniques and implementations are introduced, including the preconditioning, spline collocation, nonuniform time discretizations, etc.


\end{document}